\documentclass[12pt]{amsart}

\usepackage[colorlinks=true, pdfstartview=FitV, linkcolor=blue, citecolor=blue, urlcolor=blue, breaklinks=true]{hyperref}
\usepackage{amsmath,amsfonts,amssymb,amsthm,amscd,comment,paralist,stmaryrd,etoolbox,mathtools}
\usepackage[usenames]{color}
\usepackage{mdwlist}

\newcommand\C{\mathbb{C}}
\newcommand\Z{\mathbb{Z}}
\newcommand\Q{\mathbb{Q}}

\newcommand\N{\mathbb{N}}
\newcommand\F{\mathbb{F}}
\newcommand\kk{\Bbbk}

\newcommand\id{\mathrm{Id}}

\newcommand\fh{\mathfrak{h}}
\newcommand\fq{\mathfrak{q}}

\newcommand\cF{\mathcal{F}}
\newcommand\cG{\mathcal{G}}
\newcommand\cK{\mathcal{K}}
\newcommand\partition{\mathcal{P}}
\newcommand\comp{\mathcal{C}}

\newcommand\bh{\mathbf{h}}
\newcommand\br{\mathbf{r}}

\newcommand\Sy{\mathrm{Sym}}
\newcommand\QS{\mathrm{QSym}}
\newcommand\NS{\mathrm{NSym}}
\newcommand\op{\mathrm{op}}

\newcommand\pj{\mathrm{proj}}

\newcommand\Vect{\mathrm{Vect}}

\newcommand\Rh{\prescript{R}{}{\bh}}
\newcommand\pR[1]{\prescript{R}{}{#1}}
\newcommand\pL[1]{\prescript{L}{}{#1}}

\newcommand\ts{\textstyle}

\DeclareMathOperator{\Hom}{Hom}
\DeclareMathOperator{\End}{End}
\DeclareMathOperator{\Span}{Span}

\DeclareMathOperator{\Res}{Res}
\DeclareMathOperator{\Ind}{Ind}
\DeclareMathOperator{\md}{-mod}
\DeclareMathOperator{\pmd}{-pmod}
\DeclareMathOperator{\rad}{rad}

\newtheorem{theo}{Theorem}[section]
\newtheorem{prop}[theo]{Proposition}
\newtheorem{lem}[theo]{Lemma}
\newtheorem{cor}[theo]{Corollary}

\theoremstyle{definition}
\newtheorem{defin}[theo]{Definition}
\newtheorem{rem}[theo]{Remark}

\numberwithin{equation}{section}
\allowdisplaybreaks

\newtoggle{comments}
\newtoggle{details}
\newtoggle{prelimnote}

\toggletrue{comments}   

\togglefalse{details} 

\togglefalse{prelimnote} 

\iftoggle{comments}{%
  \newcommand{\comments}[1]{
    \begin{center}
      \parbox{6.5 in}{
        \color{red}
          {\footnotesize \textbf{Comments:} #1}
        \color{black}}
    \end{center}}
}{%
  \newcommand{\comments}[1]{}
}

\iftoggle{details}{%
  \newcommand{\details}[1]{
    \begin{center}
      \parbox{6 in}{
        \color{blue}
          {\footnotesize \textbf{Details:} #1}
        \color{black}}
    \end{center}}
}{%
  \newcommand{\details}[1]{}
}

\iftoggle{prelimnote}{%
  \newcommand{\prelim}{\textsc{Preliminary version: Do not distribute} \bigskip}
}{%
  \newcommand{\prelim}{}
}

\begin{document}
%

\title{Categorification and Heisenberg doubles arising from towers of algebras}

\author{Alistair Savage}
\address{A.~Savage: Department of Mathematics and Statistics, University of Ottawa, Canada}
\urladdr{\url{http://mysite.science.uottawa.ca/asavag2/}}
\email{alistair.savage@uottawa.ca}
\thanks{The first author was supported by a Discovery Grant from the Natural Sciences and Engineering Research Council of Canada.  The second author was supported by an AMS--Simons Travel Grant.  Part of this work was completed while the second author was at the University of Toronto.}

\author{Oded Yacobi}
\address{O.~Yacobi: School of Mathematics and Statistics, University of Sydney, Australia}
\urladdr{\url{http://www.maths.usyd.edu.au/u/oyacobi/}}
\email{oyacobi@maths.usyd.edu.au}

\begin{abstract}
  The Grothendieck groups of the categories of finitely generated modules and finitely generated projective modules over a tower of algebras can be endowed with (co)algebra structures that, in many cases of interest, give rise to a dual pair of Hopf algebras.  Moreover, given a dual pair of Hopf algebras, one can construct an algebra called the Heisenberg double, which is a generalization of the classical Heisenberg algebra.  The aim of this paper is to study Heisenberg doubles arising from towers of algebras in this manner.  First, we develop the basic representation theory of such Heisenberg doubles and show that if induction and restriction satisfy Mackey-like isomorphisms then the Fock space representation of the Heisenberg double has a natural categorification.  This unifies the existing categorifications of the polynomial representation of the Weyl algebra and the Fock space representation of the Heisenberg algebra.  Second, we develop in detail the theory applied to the tower of $0$-Hecke algebras, obtaining new Heisenberg-like algebras that we call \emph{quasi-Heisenberg algebras}.  As an application of a generalized Stone--von Neumann Theorem, we give a new proof of the fact that the ring of quasisymmetric functions is free over the ring of symmetric functions.
\end{abstract}


\prelim

\maketitle
\thispagestyle{empty}

\tableofcontents

%
\section{Introduction}
%

The interplay between symmetric groups and the Heisenberg algebra has a rich history, with implications in combinatorics, representation theory, and mathematical physics.  A foundational result in this theory is due to Geissinger, who gave a representation theoretic realization of the bialgebra of symmetric functions $\Sy$ by considering the Grothendieck groups of representations of all symmetric groups over a field $\kk$ of characteristic zero (see~\cite{Gei77}).  In particular, he constructed an isomorphism of bialgebras
\[ \ts
  \Sy \cong \bigoplus_{n=0}^\infty \cK_0(\kk[S_n]\md),
\]
where $\cK_0(\mathcal{C})$ denotes the Grothendieck group of an abelian category $\mathcal{C}$.  Multiplication is described by the induction functor
\[
  [\Ind]: \cK_0(\kk[S_n]\md) \otimes \cK_0(\kk[S_m]\md)
  \to \cK_0(\kk[S_{n+m}]\md),
\]
while comultiplication is given by restriction.  Mackey theory for induction and restriction in  symmetric groups implies that the coproduct is an algebra homomorphism.  For each $S_n$-module $V$, multiplication by the class $[V] \in \cK_0(\kk[S_n]\md)$ defines an endomorphism of $\bigoplus_{n=0}^\infty \cK_0(\kk[S_n]\md)$.  These endomorphisms, together with their adjoints, define a representation of the Heisenberg algebra on $\bigoplus_{n=0}^\infty \cK_0(\kk[S_n]\md)$.

Geissinger's construction was $q$-deformed by Zelevinsky in \cite{Zel81}, who replaced the group algebra of the symmetric group $\kk[S_n]$ by the Hecke algebra $H_n(q)$ at generic $q$.  Again, endomorphisms of the Grothendieck group given by multiplication by classes $[V]$, together with their adjoints, generate a representation of the Heisenberg algebra.

The above results can be enhanced to a categorification of the Heisenberg algebra and its Fock space representation via categories of modules over symmetric groups and Hecke algebras.  A strengthened version of this categorification, which includes information about the natural transformations involved, was given in~\cite{Kho10} for the case of symmetric groups and in~\cite{LS13} for the case of Hecke algebras.

The group algebras of symmetric groups and Hecke algebras are both examples of \emph{towers of algebras}.  A tower of algebras is a graded algebra $A=\bigoplus_{n \geq 0}A_n$, where each $A_n$ is itself an algebra (with a different multiplication) and such that the multiplication in $A$ induces homomorphisms $A_m\otimes A_n \to A_{m+n}$ of algebras (see Definition~\ref{def:tower}).  In addition to those mentioned above, examples include nilcoxeter algebras, 0-Hecke algebras, Hecke algebras at roots of unity, wreath products (semidirect products of symmetric groups and finite groups, see \cite{FJW00,CL12}), group algebras of finite general linear groups, and cyclotomic Khovanov--Lauda--Rouquier algebras (quiver Hecke algebras).  To a tower of algebras, one can associate the $\Z$-modules $\cG(A)=\bigoplus_n \cK_0(A_n\md)$ and $\cK(A)=\bigoplus_n \cK_0(A_n\pmd)$, where $A_n\md$ (respectively $A_n\pmd$) is the category of finitely generated (respectively finitely generated projective) $A_n$-modules.  In many cases, induction and restriction endow $\cK(A)$ and $\cG(A)$ with the structure of dual Hopf algebras.  For example, in \cite{BL09} Bergeron and Li introduced a set of axioms for a tower of algebras that ensure this duality (although the axioms we consider in the current paper are different).

One of the main goals of the current paper is to generalize the above categorifications of the Fock space representation of the Heisenberg algebra to more general towers of algebras.  We see that, in the general situation, the Heisenberg algebra $\fh$ is replaced by the \emph{Heisenberg double} (see Definition~\ref{def:h}) of $\cG(A)$.  The Heisenberg double of a Hopf algebra is different from, but closely related to, the more well known Drinfeld quantum double.  As a $\kk$-module, the Heisenberg double $\fh(H^+,H^-)$ of a Hopf algebra $H^+$ (over $\kk$) with dual $H^-$ is isomorphic to $H^+ \otimes_\kk H^-$, and the factors $H^-$ and $H^+$ are subalgebras.  The most well known example of this construction is when $H^-$ and $H^+$ are both the Hopf algebra of symmetric functions, which is self-dual.  In this case the Heisenberg double is the classical Heisenberg algebra.   In general, there is a natural action of $\fh(H^+,H^-)$ on its Fock space $H^+$ generalizing the usual Fock space representation of the Heisenberg algebra.   Our first result (Theorem~\ref{theo:Fock-space-properties}) is a generalization of the well known Stone--von Neumann Theorem to this Heisenberg double setting.

In the special case of dual Hopf algebras arising from a tower of algebras $A$, we denote the Heisenberg double by $\fh(A)$ and the resulting Fock space by $\cF(A)$.  In this situation, there is a natural subalgebra of $\fh(A)$.  In particular, the image $\cG_\pj(A)$ of the natural Cartan map $\cK(A) \to \cG(A)$ is a Hopf subalgebra of $\cG(A)$, and we consider also the \emph{projective} Heisenberg double ${\mathfrak{h}_\pj}(A)$ which is, by definition, the subalgebra of $\fh(A)$ generated by $\cG_\pj(A)$ and $\cK(A)$.  Then ${\mathfrak{h}_\pj}(A)$ acts on its Fock space $\cG_\pj(A)$, and a Stone--von Neumann type theorem also holds for this action (see Proposition~\ref{prop:p-Fock-space-properties}).

We then focus our attention on towers of algebras that satisfy natural compatibility conditions between induction and restriction analogous to the well known Mackey theory for finite groups.  We call these towers of algebras \emph{strong} (see Definition~\ref{def:strong}) and give a necessary and sufficient condition for them to give rise to dual pairs of Hopf algebras (i.e.\ be \emph{dualizing}).   Our central theorem (Theorem~\ref{theo:categorification}) is that, for such towers, the Fock spaces representations of the algebras $\fh(A)$ and ${\mathfrak{h}_\pj}(A)$ admit categorifications coming from induction and restrictions functors on $\bigoplus_n A_n\md$ and $\bigoplus_n A_n\pmd$ respectively.

To illustrate our main result, we apply it to several towers of algebras that are quotients of group algebras of braid groups by quadratic relations (see Definition~\ref{def:Hecke-like}).  We first show that all towers of this form are strong and dualizing.  Examples include the towers of nilcoxeter algebras, Hecke algebras, and 0-Hecke algebras.  Starting with the tower of nilcoxeter algebras, we recover Khovanov's categorification of the polynomial representation of the Weyl algebra (see Section~\ref{sec:Weyl}).  Taking instead the tower of Hecke algebras at a generic parameter, we recover (weakened versions of) the categorifications of the Fock space representation of the Heisenberg algebra described by Khovanov and Licata--Savage (see Section~\ref{sec:sym}).  Considering the tower of Hecke algebras at a root of unity, we obtain a different categorification of the Fock space representation of the Heisenberg algebra (Proposition~\ref{prop:Heis-unity}) which, in the setting of the existing categorification of the basic representation of affine $\mathfrak{sl}_n$ using this tower, corresponds to the principal Heisenberg subalgebra.  In this way, we see that Theorem~\ref{theo:categorification} provides a uniform treatment and generalization of these categorification results.  A major feature of our categorification is that it does not depend on any presentation of the algebras in question, in contrast to many categorification results in the literature.

We explore the example of the tower $A$ of $0$-Hecke algebras in some detail.  In this case, it is known that $\cK(A)$ and $\cG(A)$ are the Hopf algebras of noncommutative symmetric functions and quasisymmetric functions respectively.  However, the algebras $\fh(A)$ and ${\mathfrak{h}_\pj}(A)$, which we call the \emph{quasi-Heisenberg algebra} and \emph{projective quasi-Heisenberg algebra}, do not seem to have been studied in the literature.  We give presentations of these algebras by generators and relations (see Section~\ref{subsec:quasi-Heisenberg}).  The algebra ${\mathfrak{h}_\pj}(A)$ turns out to be particularly simple as it is a ``de-abelianization'' of the usual Heisenberg algebra (see Proposition~\ref{prop:p-presentation}).  As an application of the generalized Stone--von Neumann Theorem in this case, we give a representation theoretic proof of the fact that the ring of quasisymmetric functions is free over the ring of symmetric functions (Proposition~\ref{prop:qsym-free-over-sym}).  Our proof is quite different than previous proofs appearing in the literature.

There are many more examples of towers of algebras for which we do not work out the detailed implications of our main theorem.  Furthermore, we expect that the results of this paper could be generalized to apply to towers of superalgebras.  Examples of such towers include Sergeev algebras and 0-Hecke-Clifford algebras.  We leave such generalizations for future work.

\subsection*{Notation}

We let $\N$ and $\N_+$ denote the set of nonnegative and positive integers respectively.  We let $\kk$ be a commutative ring (with unit) and $\F$ be a field.  For $n \in \N$, we let $\partition(n)$ denote the set of all partitions of $n$, with the convention that $\partition(0) = \{\varnothing\}$, and let $\partition = \bigcup_{n \in \N} \partition(n)$.  Similarly, we let $\comp(n)$ denote the set of all compositions of $n$ and let $\comp = \bigcup_{n \in \N} \comp(n)$.  For a composition or partition $\alpha$, we let $\ell(\alpha)$ denote the length of $\alpha$ (i.e.\ the number of nonzero parts) and let $|\alpha|$ denote its size (i.e.\ the sum of its parts).  By a slight abuse of terminology, we will use the terms \emph{module} and \emph{representation} interchangeably.

\subsection*{Acknowledgements}

The authors would like to thank N.\ Bergeron, J.\ Bernstein, M.~Khovanov, A.~Lauda, A.\ Licata, C.~Reutenauer, J.-Y. Thibon, and M.\ Zabrocki for useful conversations.

%
\section{The Heisenberg double} \label{sec:general-construction}
%

In this section, we review the definition of the Heisenberg double of a Hopf algebra and state some important facts about its natural Fock space representation.  In particular, we prove a generalization of the well known Stone--von Neumann Theorem (Theorem~\ref{theo:Fock-space-properties}).

We fix a commutative ring $\kk$ and all algebras, coalgebras, bialgebras and Hopf algebras will be over $\kk$.  We will denote the multiplication, comultiplication, unit, counit and antipode of a Hopf algebra by $\nabla$, $\Delta$, $\eta$, $\varepsilon$ and $S$ respectively.  We will use Sweedler notation
\[ \ts
  \Delta(a) = \sum_{(a)} a_{(1)} \otimes a_{(2)}
\]
for coproducts.  For a $\kk$-module $V$, we will simply write $\End V$ for $\End_\kk V$.  All tensor products are over $\kk$ unless otherwise indicated.

\subsection{Dual Hopf algebras} \label{subsec:dual-hopf}

We begin by recalling the notion of dual (graded connected) Hopf algebras.

\begin{defin}[Graded connected Hopf algebra]
  We say that a bialgebra $H$ is \emph{graded} if $H = \bigoplus_{n \in \N} H_n$, where each $H_n$ is finitely generated and free as a $\kk$-module, and the following conditions are satisfied:
  \begin{gather*} \ts
    \nabla(H_k \otimes H_\ell) \subseteq H_{k + \ell},\quad \Delta(H_k) \subseteq \bigoplus_{j=0}^k H_j \otimes H_{k-j},\quad k,\ell \in \N, \\
    \eta(\kk) \subseteq H_0,\quad \varepsilon(H_k) = 0 \text{ for } k \in \N_+.
  \end{gather*}
  We say that $H$ is \emph{graded connected} if it is graded and $H_0 = \kk 1_H$.  Recall that a graded connected bialgebra is a Hopf algebra with invertible antipode (see, for example, \cite[p.~389, Cor.~5]{Hand08}) and thus we will also call such an object a \emph{graded connected Hopf algebra}.
\end{defin}

If $H = \bigoplus_{n \in \N} H_n$ is a graded bialgebra, then its \emph{graded dual} $\bigoplus_{n \in \N} H_n^*$ is also a graded bialgebra.

\begin{rem}
  In general, one need not assume that the $H_n$ are free as $\kk$-modules.  Instead, one needs only assume that
  \begin{equation} \label{eq:dual-condition}
    H_k^* \otimes H_\ell^* \cong (H_k \otimes H_\ell)^* \text{ for all } k,\ell \in \N
  \end{equation}
  in order for the graded dual to be a graded bialgebra.  However, since our interest lies mainly in dual Hopf algebras arising from towers of algebras, for which the $H_n$ are free as $\kk$-modules, we will make this assumption from the start (in which case \eqref{eq:dual-condition} is automatically satisfied).
\end{rem}

\begin{defin}[Hopf pairing]
  If $H$ and $H'$ are Hopf algebras over $\kk$, then a \emph{Hopf pairing} is a bilinear map $\langle \cdot,
  \cdot \rangle \colon H \times H' \to \kk$ such that
  \begin{gather*} \ts
    \langle ab, x \rangle = \langle a \otimes b, \Delta(x) \rangle = \sum_{(x)} \langle a, x_{(1)} \rangle \langle b, x_{(2)} \rangle, \\ \ts
    \langle a, xy \rangle = \langle \Delta(a), x \otimes y \rangle = \sum_{(a)} \langle a_{(1)}, x \rangle \langle a_{(2)}, y \rangle, \\
    \langle 1_H, x \rangle = \varepsilon(x),\quad \langle a, 1_{H'} \rangle = \varepsilon(a),
  \end{gather*}
  for all $a,b \in H$, $x,y \in H'$.  Note that such a Hopf pairing automatically satisfies $\langle a, S(x) \rangle = \langle S(a), x \rangle$ for all $a \in H$ and $x \in H'$.
\end{defin}

Recall that, for $\kk$-modules $V$ and $W$, a bilinear form $\langle \cdot, \cdot \rangle \colon V \times W \to \kk$ is called a \emph{perfect pairing} if the induced map $\Phi \colon V \to W^*$ given by $\Phi(v)(w) = \langle v, w \rangle$ is an isomorphism.

\begin{defin}[Dual pair] \label{def:dual-pair}
  We say that $(H^+, H^-)$ is a \emph{dual pair} of Hopf algebras if $H^+$ and $H^-$ are both graded connected Hopf algebras and $H^\pm$ is graded dual to $H^\mp$ (as a Hopf algebra) via a perfect Hopf pairing $\langle \cdot, \cdot \rangle \colon H^- \times H^+ \to \kk$.
\end{defin}

\subsection{The Heisenberg double}
\label{subsec:h-definition}

For the remainder of this section, we fix a dual pair $(H^+,H^-)$ of Hopf algebras.  Then any $a \in H^+$ defines an element $\pL{a} \in \End H^+$ by left multiplication.  Similarly, any $x \in H^-$ defines an element $\pR{x} \in \End H^- $ by right multiplication, whose adjoint $\pR{x}^*$ is an element of $\End H^+$.  (In the case that $H^+$ or $H^-$ is commutative, we often omit the superscript $L$ or $R$.)  In this way we have $\kk$-algebra homomorphisms
\begin{gather}
  H^+ \hookrightarrow \End H^+,\quad a \mapsto \pL{a}, \label{eq:H+action} \\
  H^- \hookrightarrow \End H^+,\quad x \mapsto \pR{x}^*. \label{eq:H-action}
\end{gather}
The action of $H^-$ on $H^+$ given by~\eqref{eq:H-action} is called the \emph{left-regular action}.

\begin{lem} \label{lem:left-reg-action}
  The left-regular action of $H^-$ on $H^+$ is given by
  \[ \ts
    \pR{x}^*(a) = \sum_{(a)} \langle x, a_{(2)} \rangle a_{(1)} \quad \text{for all } x \in H^-,\ a \in H^+.
  \]
\end{lem}

\begin{proof}
  For all $x,y \in H^-$ and $a \in H^+$, we have
  \begin{multline*} \ts
    \langle y, \pR{x}^*(a) \rangle = \langle yx, a \rangle = \langle y \otimes x, \Delta(a) \rangle \\ \ts
    = \sum_{(a)} \langle y \otimes x, a_{(1)} \otimes a_{(2)} \rangle = \sum_{(a)} \langle y, a_{(1)} \rangle \langle x, a_{(2)} \rangle = \langle y, \sum_{(a)} \langle x, a_{(2)} \rangle a_{(1)} \rangle.
  \end{multline*}
  The result then follows from the nondegeneracy of the bilinear form.
\end{proof}

It is clear that the map~\eqref{eq:H+action} is injective.  The map~\eqref{eq:H-action} is also injective.  Indeed, for $x \in H^-$, choose $a \in H^+$ such that $\langle x, a \rangle \ne 0$.  Then $\langle 1, \pR{x}^*(a) \rangle = \langle x, a \rangle \ne 0$, and so $\pR{x}^* \ne 0$.

Since $H^+ = \bigoplus_{n \in \N} H_n^+$ is $\N$-graded, we have a natural algebra $\Z$-grading $\End H^+ = \bigoplus_{n \in \Z} \End_n H^+$.  It is routine to verify that the map~\eqref{eq:H+action} sends $H^+_n$ to $\End_n H^+$ and the map~\eqref{eq:H-action} sends $H^-_n$ to $\End_{-n} H^+$ for all $n \in \N$.

\begin{defin}[The Heisenberg double, {\cite[Def.~3.1]{STS94}}] \label{def:h}
  We define $\fh(H^+,H^-)$ to be the \emph{Heisenberg double} of $H^+$.  More precisely $\fh(H^+,H^-) \cong H^+ \otimes H^-$ as $\kk$-modules, and we write $a \# x$ for $a \otimes x$, $a \in H^+$, $x \in H^-$, viewed as an element of $\fh(H^+,H^-)$.  Multiplication is given by
  \begin{equation} \ts
    (a \# x)(b \# y) := \sum_{(x)} a \pR{x_{(1)}}^*(b) \# x_{(2)}y = \sum_{(x),(b)} \langle x_{(1)}, b_{(2)} \rangle ab_{(1)} \# x_{(2)} y.
  \end{equation}
  We will often view $H^+$ and $H^-$ as subalgebras of $\fh(H^+,H^-)$ via the maps $a \mapsto a \# 1$ and $x \mapsto 1 \# x$ for $a \in H^+$ and $x \in H^-$.  Then we have $ax = a \# x$.  When the context is clear, we will simply write $\fh$ for $\fh(H^+,H^-)$.  We have a natural grading $\fh = \bigoplus_{n \in \Z} \fh_n$, where $\fh_n = \bigoplus_{k - \ell = n} H_k^+ \# H_\ell^-$.
\end{defin}

\begin{rem}
  The Heisenberg double is a twist of the Drinfeld quantum double by a right 2-cocycle (see \cite[Th.~6.2]{Lu94}).
\end{rem}

\begin{lem} \label{lem:adjoint-action-on-product}
  If $x \in H^-$ and $a,b \in H^+$, then
  \[ \ts
    \pR{x}^*(ab) = \sum_{(x)} \pR{x_{(1)}}^*(a) \pR{x_{(2)}}^*(b).
  \]
\end{lem}

\begin{proof}
  For $x,y \in H^-$ and $a,b \in H^+$, we have
  \begin{multline*} \ts
    \langle y, \pR{x}^*(ab) \rangle = \langle yx, ab \rangle = \langle \Delta(yx), a \otimes b \rangle = \langle \Delta(y) \Delta(x), a \otimes b \rangle \\ \ts
    = \langle \Delta(y), \pR{\Delta(x)}^*(a \otimes b) \rangle = \langle \Delta(y), \sum_{(x)} \pR{x}_{(1)}^*(a) \otimes \pR{x}_{(2)}^*(b) \rangle = \langle y, \sum_{(x)}  \pR{x}_{(1)}^*(a) \pR{x}_{(2)}^*(b) \rangle.
  \end{multline*}
  The result then follows from the nondegeneracy of the bilinear form.
\end{proof}

Interchanging $H^-$ and $H^+$ in the construction of the Heisenberg double results in the opposite algebra: $\fh(H^-,H^+) \cong \fh(H^+,H^-)^\op$ (see~\cite[Prop.~5.3]{Lu94}).

\subsection{Fock space} \label{subsec:general-Fock-space}

We now introduce a natural representation of the algebra $\fh$.

\begin{defin}[Vacuum vector]
  An element $v$ of an $\fh$-module $V$ is called a \emph{lowest weight} (resp.\ \emph{highest weight}) \emph{vacuum vector} if $\kk v \cong \kk$ and $H^- v = 0$ (resp.\ $H^+ v = 0$).
\end{defin}

\begin{defin}[Fock space]
  The algebra $\fh$ has a natural (left) representation on $H^+$ given by
  \[
    (a \# x)(b) = a \pR{x}^*(b),\quad a,b \in H^+,\ x \in H^-.
  \]
  We call this the \emph{lowest weight Fock space representation} of $\fh(H^+,H^-)$ and denote it by $\cF = \cF(H^+,H^-)$.  Note that this representation is generated by the lowest weight vacuum vector $1 \in H^+$.
\end{defin}

Suppose $X^+$ is a subalgebra of $H^+$ that is invariant under the left-regular action of $H^-$ on $H^+$.  (Note that it follows that $X^+$ is a graded subalgebra of $H^+$.)  Then $X^+ \# H^-$ is a subalgebra of $\fh(H^+,H^-)$ acting naturally on $X^+$.  The following result (when $X^+=H^+$) is a generalization of the Stone--von Neumann Theorem to the setting of an arbitrary Heisenberg double.

\begin{theo} \label{theo:Fock-space-properties}
  Let $X^+$ be a subalgebra of $H^+$ that is invariant under the left-regular action of $H^-$ on $H^+$.
  \begin{enumerate}
    \item \label{theo-item:Fock-space-subreps} The only $(X^+ \# H^-)$-submodules of $X^+$ are those of the form $I X^+$ for some ideal $I$ of $\kk$.  In particular, if $\kk$ is a field, then $X^+$ is irreducible as an $(X^+ \# H^-)$-module.

    \item Let $\kk^- \cong \kk$ (isomorphism of $\kk$-modules) be the representation of $H^-$ such that $H^-_n$ acts as zero for all $n > 0$ and $H^-_0 \cong \kk$ acts by left multiplication.  Then $X^+$ is isomorphic to the induced module $\Ind^{X^+ \# H^-}_{H^-} \kk^-$ as an $(X^+ \# H^-)$-module.

    \item \label{theo-item:Stone-von-Neumann} Any $(X^+ \# H^-)$-module generated by a lowest weight vacuum vector is isomorphic to $X^+$.
  \suspend{enumerate}
  If $X^+=H^+$ then $X^+ \# H^- = \fh(H^+,H^-)$ and the module $X^+$ is the lowest weight Fock space $\cF$.  In that case we also have the following.
  \resume{enumerate}
    \item \label{theo-item:Fock-space-faithful} The lowest weight Fock space representation $\cF$ of $\fh$ is faithful.
  \end{enumerate}
\end{theo}

\begin{proof}
  Throughout the proof, we write $1$ for $1_{H^+} = 1_{X^+}$.
  \begin{asparaenum}
    \item Clearly, if $I$ is an ideal of $\kk$, then $I X^+$ is a submodule of $X^+$.  Now suppose $W \subseteq X^+$ is a nonzero submodule, and let
        \[
          I = \{c \in \kk\ |\ c1 \in W\}.
        \]
        It is easy to see that $I$ is an ideal of $\kk$.  We claim that $W = I X^+$.  Since the element $1$ generates $X^+$, we clearly have $I X^+ \subseteq W$.  Now suppose there exists $a \in W$ such that $a \not \in I X^+$. Without loss of generality, we can write $a = \sum_{n=0}^\ell a_n$ for $a_n \in H_n^+$, $\ell \in \N$, and $a_\ell \not \in I X^+$ (otherwise, consider $a-a_\ell$).  Let $b_1,\dotsc,b_m$ be a basis of $H_\ell^+$ such that $b_1,\dotsc,b_k$ is a basis of $X_\ell^+$, for $k = \dim_\kk X_\ell^+$.  Let $x_1,\dotsc,x_m$ be the dual basis of $H_\ell^-$.  Then it is easy to verify that $\sum_{j=1}^k b_j \# x_j$ acts as the identity on $X_\ell^+$ and as zero on $X_n^+$ for $n < \ell$. Thus,
        \[ \ts
          a_\ell = \sum_{j=1}^k (b_j \# x_j)(a) \in \sum_{j=1}^k b_j I X^+ \subseteq I X^+,
        \]
        since $\pR{x_j}^*(a) \in W \cap H_0^+$ for all $j=1,\dotsc,k$.  This contradiction completes the proof.

    \item We have an injective homomorphism of $H^-$-modules
        \[
          \kk^- \hookrightarrow \Res^{X^+ \# H^-}_{H^-} X^+, \quad 1 \mapsto 1.
        \]
        Since induction is left adjoint to restriction (see, for example, \cite[(2.19)]{CR81}), this gives rise to a homomorphism of $\fh$-modules
        \begin{equation} \label{eq:Ind-to-F}
          \Ind^{X^+ \# H^-}_{H^-} \kk^- \to X^+,\quad 1 \mapsto 1.
        \end{equation}
        Since the element $1$ generates $X^+$, this map is surjective.  Now,
        \[
          \Ind^{X^+ \# H^-}_{H^-} \kk^- = (X^+ \# H^-) \otimes_{H^-} \kk^- = X^+ H^- \otimes_{H^-} \kk^-.
        \]
        It follows that, as a left $X^+$-module, $\Ind^{X^+ \# H^-}_{H^-} \kk^-$ is a quotient of $X^+ \otimes_\kk \kk^- \cong X^+$ and the map~\eqref{eq:Ind-to-F} is the identity map, hence an isomorphism.

    \item Suppose $V$ is a representation of $X^+ \# H^-$ generated by a lowest weight vacuum vector $v_0$.  Then, as above, we have an injective homomorphism of $H^-$-modules
        \[
          \kk^- \hookrightarrow \Res^{X^+ \# H^-}_{H^-} V, \quad 1 \mapsto v_0,
        \]
        and thus a homomorphism of $(X^+ \# H^-)$-modules
        \begin{equation}
          \Ind^{X^+ \# H^-}_{H^-} \kk^- \to V,\quad 1 \mapsto v_0.
        \end{equation}
        Since $V$ is generated by $v_0$, this map is surjective.  Since $\Ind^{X^+ \# H^-}_{H^-} \kk^- \cong X^+$, it follows easily from part~\eqref{theo-item:Fock-space-subreps} that it is also injective.

    \item Suppose $\alpha$ is a nonzero element of $\fh$.  Write $\alpha = \alpha' + \alpha''$ where $\alpha'$ is a nonzero element of $H^+ \# H^-_n$ for some $n \in \N$ and $\alpha'' \in \sum_{k>n} H^+ \# H^-_k$. Choose a basis $x_1,\dotsc,x_m$ of $H_n^-$ and let $b_1,\dotsc,b_m$ denote the dual basis of $H_n^+$.  Then we can write $\alpha' = \sum_{j=1}^m a_j \# x_j$ for some $a_j \in H^+$.  Since $\alpha' \ne 0$, we have $a_j \ne 0$ for some $j$.  Then $\alpha(b_j) = \alpha'(b_j) = a_j \ne 0$.  Thus the action of $\fh$ on $\cF$ is faithful. \qedhere
  \end{asparaenum}
\end{proof}

\begin{rem}
  By Theorem~\ref{theo:Fock-space-properties}\eqref{theo-item:Fock-space-faithful}, we may view $\fh$ as the subalgebra of $\End H^+$ generated by $\pL{a}$, $a \in H^+$, and $\pR{x}^*$, $x \in H^-$.
\end{rem}

%
\section{Towers of algebras and categorification of Fock space} \label{sec:dual-pairs-from-towers}
%

In this section, we consider dual Hopf algebras arising from towers of algebras.  In this case, we are able to deduce some further results about the Heisenberg double $\fh$.  We then prove our main result, that towers of algebras give rise to categorifications of the lowest weight Fock space representation of $\fh$ (Theorem~\ref{theo:categorification}).  Recall that $\F$ is an arbitrary field.

\subsection{Modules categories and their Grothendieck groups} \label{subsec:module-categories}

For an arbitrary $\F$-algebra $B$, let $B\md$ denote the category of finitely generated left $B$-modules and let $B\pmd$ denote the category of finitely generated projective left $B$-modules.  We then define
\[
  G_0(B) = \cK_0(B\md) \quad \text{and} \quad K_0(B) = \cK_0(B\pmd),
\]
where $\cK_0(B\md)$ denotes the Grothendieck group of the abelian category $B\md$, and $\cK_0(B\pmd)$ denotes the (split) Grothendieck group of the additive category $B\pmd$.  For $\mathcal{C} = B\md$ or $B\pmd$, we denote the class of an object $M \in \mathcal{C}$ in $\cK_0(\mathcal{C})$ by $[M]$.

There is a natural bilinear form
\begin{equation} \label{eq:KG-innerprod}
  \langle \cdot, \cdot \rangle \colon K_0(B) \otimes G_0(B) \to \Z,\quad \langle [P], [M] \rangle = \dim_\F \Hom_B(P,M).
\end{equation}
If $B$ is a finite dimensional algebra, let $V_1,\dotsc,V_s$ be a complete list of nonisomorphic simple $B$-modules.  If $P_i$ is the projective cover of $V_i$ for $i=1,\dotsc,s$, then $P_1,\dotsc,P_s$ is a complete list of nonisomorphic indecomposable projective $B$-modules (see, for example, \cite[Cor.~I.4.5]{ARS95}) and we have
\[ \ts
  G_0(B) = \bigoplus_{i=1}^s \Z[V_i] \quad \text{and} \quad K_0(B) = \bigoplus_{i=1}^s \Z[P_i].
\]
If $\F$ is algebraically closed, then
\begin{equation} \label{eq:pairing-simple-proj}
  \langle [P_i], [V_j] \rangle = \delta_{ij} \quad \text{for } 1 \le i,j \le s,
\end{equation}
and so the pairing~\eqref{eq:KG-innerprod} is perfect.

Suppose $\varphi \colon B \to A$ is an algebra homomorphism.  Then we can consider $A$ as a left $B$-module via the action $(b, a) \mapsto \varphi(b) a$.  Similarly, we can consider $A$ as a right $B$-module.  Then we have \emph{induction} and \emph{restriction} functors
\begin{gather*}
  \Ind^A_B \colon B\md \to A\md,\quad \Ind^A_B N := A \otimes_B N,\quad N \in B\md, \\
  \Res^A_B \colon A\md \to B\md,\quad \Res^A_B M := \Hom_A(A,M) \cong {_BA} \otimes_A M,\quad M \in A\md,
\end{gather*}
where $_BA$ denotes $A$ considered as a $(B,A)$-bimodule and the left $B$-action on $\Hom_A(A,M)$ is given by $(b, f) \mapsto f \circ \pR{b}$ for $f \in \Hom_A(A,M)$ and $b \in B$ (here $\pR{b}$ denotes right multiplication by $b$).  The isomorphism above is given by the map $f \mapsto 1 \otimes f(1_A)$ for $f \in \Hom_A(A,M)$.  This isomorphism is natural in $M$ and so we have an isomorphism of functors $\Res^A_B \cong {_BA} \otimes_A -$.

\subsection{Towers of algebras} \label{subsec:towers}

\begin{defin}[Tower of algebras] \label{def:tower}
  Let $A = \bigoplus_{n \in \N} A_n$ be a graded algebra over a field $\F$ with multiplication $\rho \colon A \otimes A \to A$.  Then $A$ is called a \emph{tower of algebras} if the following conditions are satisfied:
  \begin{enumerate}
    \item[(TA1)] Each graded piece $A_n$, $n \in \N$, is a finite dimensional algebra (with a different multiplication) with a unit $1_n$.  We have $A_0 \cong \F$.
    \item[(TA2)] The external multiplication $\rho_{m,n} \colon A_m \otimes A_n \to A_{m+n}$ is a homomorphism of algebras for all $m,n \in \N$ (sending $1_m \otimes 1_n$ to $1_{m+n}$).
    \item[(TA3)] We have that $A_{m+n}$ is a two-sided projective $(A_m \otimes A_n)$-module with the action defined by
        \[
          a \cdot (b \otimes c) = a \rho_{m,n}(b \otimes c) \quad \text{and} \quad (b \otimes c) \cdot a = \rho_{m,n}(b \otimes c) a,
        \]
        for all $m,n \in \N$, $a \in A_{m+n}$, $b \in A_m$, $c \in A_n$.
    \item[(TA4)] For each $n \in \N$, the pairing~\eqref{eq:KG-innerprod} (with $B=A_n$) is perfect.  (Note that this condition is automatically satisfied if $\F$ is an algebraically closed field, by~\eqref{eq:pairing-simple-proj}.)
    \end{enumerate}
\end{defin}

For the remainder of this section we assume that $A$ is a tower of algebras.  We let
\begin{equation} \label{eq:G0A-def} \ts
  \cG(A) = \bigoplus_{n \in \N} G_0(A_n) \quad \text{and} \quad \cK(A) = \bigoplus_{n \in \N} K_0(A_n).
\end{equation}
Then we have a perfect pairing $\langle \cdot, \cdot \rangle \colon \cK(A) \times \cG(A) \to \Z$ given by
\begin{equation} \label{eq:tower-pairing}
  \langle [P], [M] \rangle =
  \begin{cases}
    \dim_\F (\Hom_{A_n}(P,M)) & \text{if } P \in A_n\pmd \text{ and } M \in A_n\md \text{ for some } n \in \N, \\
    0 & \text{otherwise}.
  \end{cases}
\end{equation}
We also define a perfect pairing $\langle \cdot, \cdot \rangle \colon (\cK(A) \otimes \cK(A)) \times (\cG(A) \otimes \cG(A))$ by
\[
  \langle [P] \otimes [Q], [M] \otimes [N] \rangle =
  \begin{cases}
    \dim_\F (\Hom_{A_k \otimes A_\ell}(P \otimes Q, M \otimes N)) \text{ if } P \in A_k\pmd,\ Q \in A_\ell\pmd \\
    \quad \text{ and } M \in A_k\md,\ N \in A_\ell\md \text{ for some } k,\ell \in \N, \\
    0 \text{ otherwise}.
  \end{cases}
\]
Thus we have $\langle [P] \otimes [Q], [M] \otimes [N] \rangle = \langle [P], [M] \rangle \langle [Q], [N] \rangle$.

Consider the direct sums of categories
\[ \ts
  A\md := \bigoplus_{n \in \N} A_n\md,\quad A\pmd := \bigoplus_{n \in \N} A_n\pmd.
\]
For $r \in \N_+$, we define
\begin{gather*} \ts
  A\md^{\otimes r} := \bigoplus_{n_1,\dotsc,n_r \in \N} (A_{n_1} \otimes \dotsb \otimes A_{n_r})\md,\\ \ts
  A\pmd^{\otimes r} := \bigoplus_{n_1,\dotsc,n_r \in \N} (A_{n_1} \otimes \dotsb \otimes A_{n_r})\pmd.
\end{gather*}
Then, for $i,j \in \{1,\dotsc,r\}$, $i < j$, we define $S_{ij} \colon A\md^{\otimes r} \to A\md^{\otimes r}$ to be the endofunctor that interchanges the $i$th and $j$th factors, that is, the endofunctor arising from the isomorphism
\[
  A_{n_1} \otimes \dotsb \otimes A_{n_r} \cong A_{n_1} \otimes \dotsb \otimes A_{n_{i-1}} \otimes A_{n_j} \otimes A_{n_{i+1}} \otimes \dotsb \otimes A_{n_{j-1}} \otimes A_{n_i} \otimes A_{n_{j+1}} \otimes \dotsb \otimes  A_{n_r}.
\]
We use the same notation to denote the analogous endofunctor on $A\pmd^{\otimes r}$.

We also have the following functors:
\begin{equation} \label{eq:Groth-operations}
  \begin{split}
  \nabla \colon A\md^{\otimes 2} \to A\md,\quad \nabla|_{(A_m \otimes A_n)\md} = \Ind^{A_{m + n}}_{A_m \otimes A_n}, \\ \ts
  \Delta \colon A\md \to A\md^{\otimes 2}, \quad \Delta|_{A_n\md} = \bigoplus_{k+\ell =n } \Res^{A_n}_{A_k \otimes A_\ell}, \\
  \eta \colon \Vect \to A\md,\quad \eta(V) = V \in A_0\md \text{ for } V \in \Vect, \\
  \varepsilon \colon A\md \to \Vect,\quad \varepsilon(V) =
  \begin{cases}
    V & \text{if } V \in A_0\md, \\
    0 & \text{otherwise}.
  \end{cases}
  \end{split}
\end{equation}
In the above, we have identified $A_0\md$ with the category $\Vect$ of finite dimensional vector spaces over $\F$.  Replacing $A\md$ by $A\pmd$ above, we also have the functors $\nabla$, $\Delta$, $\eta$ and $\varepsilon$ on $A\pmd$.  Since the above functors are all exact (we use axiom (TA3) here), they induce a multiplication, comultiplication, unit and counit on $\cG(A)$ and $\cK(A)$.  We use the same notation to denote these induced maps.

Since induction is always left adjoint to restriction (see, for example, \cite[(2.19)]{CR81}), $\nabla$ is left adjoint to $\Delta$.  However, in many examples of towers of algebras (e.g.\ nilcoxeter algebras and 0-Hecke algebras), induction is not right adjoint to restriction.  Nevertheless, we often have something quite close to this property.  Any algebra automorphism $\psi_n$ of $A_n$ induces an isomorphism of categories $\Psi_n \colon A_n\md \to A_n\md$ (which restricts to an isomorphism of categories $\Psi_n \colon A_n\pmd \to A_n\pmd$) by twisting the $A_n$ action.  Then $\Psi := \bigoplus_{n \in \N} \Psi_n$ is an automorphism of the categories $A\md$ and $A\pmd$.  It induces automorphisms (which we also denote by $\Psi$) of $\cG(A)$ and $\cK(A)$.

\begin{defin}[Conjugate adjointness] \label{def:conjugate-adjoint}
  Given a tower of algebras $A$, we say that induction is \emph{conjugate right adjoint} to restriction (and restriction is \emph{conjugate left adjoint} to induction) with \emph{conjugation} $\Psi$ if there are isomorphisms of algebras $\psi_n \colon A_n \to A_n$, $n \in \N$,  such that $\nabla$ is right adjoint to $\Psi^{\otimes 2} \Delta \Psi^{-1}$.
\end{defin}

The following lemma will be useful since many examples of towers of algebras are in fact composed of Frobenius algebras.  We refer the reader to~\cite{SY11} for background on Frobenius algebras.

\begin{lem} \label{lem:Frobenius-conjugate}
  If each $A_n$, $n \in \N$, is a Frobenius algebra, then induction is conjugate right adjoint to restriction, with conjugation given by $\psi_n$ being the inverse of the Nakayama automorphism of $A_n$.
\end{lem}

\begin{proof}
  This follows from~\cite[Lem.~1]{Kho01} by taking $B_1 = A_{m+n}$, $B_2=A_m \otimes A_n$ and $N$ to be $A_{m+n}$, considered as an $(A_{m+n},A_m \otimes A_n)$-bimodule in the natural way.
\end{proof}

\begin{defin}[Strong tower of algebras] \label{def:strong}
  We say that a tower of algebras $A$ is \emph{strong} if induction is conjugate right adjoint to restriction and we have an isomorphism of functors
  \begin{equation} \label{eq:strong-isom} \ts
    \Delta \nabla \cong \nabla^{\otimes 2} S_{23} \Delta^{\otimes 2}.
  \end{equation}
\end{defin}

\begin{rem}
  The isomorphism~\eqref{eq:strong-isom} is a compatibility between induction and restriction that is analogous to the well known Mackey theory for finite groups.  It implies that $\cK(A)$ and $\cG(A)$ are Hopf algebras under the operations defined above (see~\cite[Th.~3.5]{BL09} -- although the authors of that paper work over $\C$ and they assume that the external multiplication maps $\rho_{m,n}$ are injective, the arguments hold in the more general setting considered here).
\end{rem}

\begin{defin}[Dualizing tower of algebras]
  We say that a tower of algebras $A$ is \emph{dualizing} if, under the operations defined above, $\cK(A)$ and $\cG(A)$ are Hopf algebras which are dual, in the sense of Definition~\ref{def:dual-pair} (with $\kk = \Z$), under the bilinear form~\eqref{eq:tower-pairing} (i.e.\ the perfect pairing~\eqref{eq:tower-pairing} is a Hopf pairing).
\end{defin}

\begin{prop} \label{prop:dualizing-conjugation-isom}
  Suppose $A$ is a strong tower of algebras with conjugation $\Psi$.  Then the following statements are equivalent.
  \begin{enumerate}
    \item The tower $A$ is dualizing.

    \item \label{eq:conjugate-coproduct-identity} We have $\Psi^{\otimes 2} \Delta \Psi^{-1} (P) \cong \Delta(P)$ for all $P \in A\pmd$.

    \item \label{eq:conjugate-coproduct-identity-classes} We have $\Psi^{\otimes 2} \Delta \Psi^{-1} = \Delta$ as maps $\cK(A) \to \cK(A) \otimes \cK(A)$.
  \end{enumerate}
  In particular $A$ is dualizing if each $A_n$ is a symmetric algebra (i.e.\ if $\Psi = \id$) or, more generally, if $\Psi$ acts trivially on $\cK(A)$.
\end{prop}

\begin{proof}
  First assume that~\eqref{eq:conjugate-coproduct-identity} holds.  With one exception, the proof that $A$ is dualizing then proceeds exactly as in the proof of~\cite[Th.~3.6]{BL09} since~\eqref{eq:strong-isom} implies axiom~(5) in~\cite[\S3.1]{BL09}. (Although the authors of that paper work over $\C$ and they assume that the external multiplication maps $\rho_{m,n}$ are injective, the arguments hold in the more general setting considered here.)  The one exception is in the proof that
  \begin{equation} \label{eq:BLproofeq}
    \langle \Delta ([P]), [M] \otimes [N] \rangle = \langle [P], \nabla ([M] \otimes [N]) \rangle,
  \end{equation}
  for all $M \in A_m\md$, $N \in A_n\md$, $P \in A_{m+n}\pmd$.  However, under our assumptions, we have
  \[
    \Hom_{A_{m+n}} (P, \nabla(M \otimes N)) \cong \Hom_{A_m \otimes A_n} (\Psi^{\otimes 2} \Delta \Psi^{-1}(P), M \otimes N) \cong \Hom_{A_m \otimes A_n} (\Delta(P), M \otimes N),
  \]
  which immediately implies~\eqref{eq:BLproofeq}.  Thus $A$ is dualizing.

  Now suppose $A$ is dualizing.  Then, for all $P \in A\pmd$ and $M \in A\md^{\otimes 2}$, we have
  \[
    \langle \Psi^{\otimes 2} \Delta \Psi^{-1}([P]), [M] \rangle = \langle [P], \nabla([M]) \rangle = \langle \Delta([P]), [M] \rangle,
  \]
  where the first equality holds by the assumption that induction is conjugate right adjoint to restriction and the second equality holds by the assumption that the tower is dualizing.  Then~\eqref{eq:conjugate-coproduct-identity-classes} follows from the nondegeneracy of the bilinear form.

  The fact that~\eqref{eq:conjugate-coproduct-identity} and~\eqref{eq:conjugate-coproduct-identity-classes} are equivalent follows from the fact that every short exact sequence of projective modules splits.  Thus, for $P,Q \in A\pmd$, we have $[P] = [Q]$ in $\cK(A)$ if and only if $P \cong Q$.
\end{proof}

\begin{rem}
  It is crucial in~\eqref{eq:conjugate-coproduct-identity} that $P$ be a projective module.  The isomorphism does not hold, in general, for arbitrary modules, even if the tower is dualizing.  For instance, the tower of 0-Hecke algebras is dualizing (see Corollary~\ref{cor:Hecke-like-strong-dualizing}), but one can show that $\Psi^{\otimes 2} \Delta \Psi^{-1} \cong S_{12} \Delta$ (see Lemma~\ref{lem:Hecke-like-coprod-conjugation}).  Then~\eqref{eq:conjugate-coproduct-identity} corresponds to the fact that the comultiplication on $\NS$ is cocommutative.  However, the comultiplication on $\QS$ is not cocommutative and thus~\eqref{eq:conjugate-coproduct-identity} does not hold, in general, if $P$ is not projective.  We refer the reader to Section~\ref{sec:qsym-nsym} for further details on the tower of 0-Hecke algebras.
\end{rem}

\subsection{The Heisenberg double associated to a tower of algebras}

In this section we apply the constructions of Section~\ref{sec:general-construction} to the dual pair $(\cG(A),\cK(A))$ arising from a dualizing tower of algebras $A$.  We also see that some natural subalgebras of the Heisenberg double arise in this situation.

\begin{defin}[$\fh(A)$, $\cF(A)$, $\cG_\pj(A)$] \label{def:G-proj}
  Suppose $A$ is a dualizing tower of algebras.  We let $\fh(A) = \fh(\cG(A),\cK(A))$ and $\cF(A) = \cF(\cG(A),\cK(A))$.  For each $n \in \N$, $A_n\pmd$ is a full subcategory of $A_n\md$.  The inclusion functor induces the \emph{Cartan map} $\cK(A) \to \cG(A)$.  Let $\cG_\pj(A)$ denote the image of the Cartan map.
\end{defin}

For the remainder of this section, we fix a dualizing tower of algebras $A$ and let
\begin{equation} \label{eq:tower-notation}
  H^- = \cK(A),\ H^+ = \cG(A),\ H^+_\pj = \cG_\pj(A),\ \fh = \fh(A),\ \mathcal{F}=\mathcal{F}(A).
\end{equation}
To avoid confusion between $\cG(A)$ and $\cK(A)$, we will write $[M]_+$ to denote the class of a finitely generated (possibly projective) $A_n$-module in $H^+$ and $[M]_-$ to denote the class of a finitely generated projective $A_n$-module in $H^-$.  If $P \in A_p\pmd$ and $N \in (A_{n-p} \otimes A_p)\md$, then we have a natural $A_{n-p}$-module structure on $\Hom_{A_p}(P,N)$ given by
\begin{equation} \label{eq:hom-restriction-action}
  (a \cdot f)(b) = (a \otimes 1) \cdot (f(b)),\quad a \in A_{n-p},\ f \in \Hom_{A_p}(P,N),\ b \in P.
\end{equation}

\begin{lem} \label{lem:adjoint-action-explicit}
  If $p,n \in \N$, $P \in A_p\pmd$ and $N \in A_n\md$, then we have
  \[
    [P]_- \cdot [N]_+ =
    \begin{cases}
      0 & \text{if } p > n, \\
      [\Hom_{A_p}(P,\Res^{A_n}_{A_{n-p} \otimes A_p} N)]_+ & \text{if } p \le n.
    \end{cases}
  \]
  Here $\cdot$ denotes the action of $\fh$ on $\cF$ and $\Hom_{A_p}(P,\Res^{A_n}_{A_{n-p} \otimes A_p} N)$ is viewed as an $A_{n-p}$-module as in~\eqref{eq:hom-restriction-action}.
\end{lem}

\begin{proof}
  The case $p > n$ follows immediately from the fact that $H^+_{n-p}=0$ if $p>n$.  Assume $p \le n$.  For $R \in A_{n-p}\pmd$, we have
  \begin{align*}
    \langle [R]_-, [P]_- \cdot [N]_+ \rangle &= \langle [R]_- [P]_-, [N]_+ \rangle \\
    &= \langle \nabla([R]_- \otimes [P]_-), [N]_+ \rangle \\
    &= \langle [R]_- \otimes [P]_-, \Delta ([N]_+) \rangle \\
    &= \dim_\F\Hom_{A_{n-p} \otimes A_p} (R \otimes P, \Res^{A_n}_{A_{n-p} \otimes A_p} N) \\
    &= \dim_\F \Hom_{A_{n-p}} (R, \Hom_{A_p}(P,\Res^{A_n}_{A_{n-p} \otimes A_p} N)) \\
    &= \langle [R]_- ,[\Hom_{A_p}(P,\Res^{A_n}_{A_{n-p} \otimes A_p} N)]_+ \rangle.
  \end{align*}
  The result then follows from the nondegeneracy of the bilinear form.
\end{proof}

\begin{lem} \label{lem:hom-proj-proj}
  Suppose $\kk$ is a commutative ring and $R$ and $S$ are $\kk$-algebras.  Furthermore, suppose that $P$ is a finitely generated projective $S$-module and $Q$ is a finitely generated projective $(R \otimes S)$-module.  Then $\Hom_S (P,Q)$ is a finitely generated projective $R$-module.
\end{lem}

\begin{proof}
  If $P$ is a projective $S$-module and $Q$ is a projective $(R \otimes S)$-module then there exist $s,t \in \N$, an $S$-module $P'$ and an $(R \otimes S)$-module $Q'$ such that $P \oplus P' \cong S^s$ and $Q \oplus Q' \cong (R \otimes S)^t$.  Then we have
  \begin{multline*}
    \Hom_S(P,Q) \oplus \Hom_S(P,Q') \oplus \Hom_S(P',(R \otimes S)^t) \\
    \cong \Hom_S (S^s,(R \otimes S)^t) \cong \Hom_S(S,R \otimes S)^{st} \cong R^{st}.
  \end{multline*}
  Thus $\Hom_S(P,Q)$ is a projective $R$-module.
\end{proof}

\begin{prop} \label{prop:Hpj-invariance}
  We have that $H^+_\pj$ is a subalgebra of $H^+$ that is invariant under the left-regular action of $H^-$.
\end{prop}

\begin{proof}
  Assume $P \in A_p\pmd$ and $Q \in A_q\pmd$.  As in the proof of \cite[Prop.~3.2]{BL09}, we have that $\Ind^{A_{p+q}}_{A_p \otimes A_q} (P \otimes Q)$ is a projective $A_{p+q}$-module.  It follows that $H^+_\pj$ is a subalgebra of $H^+$.

  By Lemma~\ref{lem:adjoint-action-explicit}, it remains to show that
  \[
    \Hom_{A_p}(P, \Res^{A_q}_{A_{q-p} \otimes A_p} Q) \in A_{q-p}\pmd.
  \]
  Again, as in the proof of \cite[Prop.~3.2]{BL09}, we have that $\Res^{A_q}_{A_{q-p} \otimes A_p} Q$ is a projective $(A_{q-p} \otimes A_p)$-module.  The result then follows from Lemma~\ref{lem:hom-proj-proj}.
\end{proof}

\begin{defin}[The projective Heisenberg double ${\mathfrak{h}_\pj}(A)$] \label{def:p}
  By Proposition~\ref{prop:Hpj-invariance}, ${\mathfrak{h}_\pj}={\mathfrak{h}_\pj}(A) := H^+_\pj \# H^-$ is a subalgebra of $\fh$.  In other words, ${\mathfrak{h}_\pj}$ is the subalgebra of $\fh$ generated by $H^+_\pj$  and $H^-$ (viewing the latter two as $\Z$-submodules of $\fh$ as in Definition~\ref{def:h}).  We call ${\mathfrak{h}_\pj}$ the \emph{projective Heisenberg double} associated to $A$.
\end{defin}

\begin{defin}[Fock space $\cF_\pj(A)$ of ${\mathfrak{h}_\pj}$] \label{def:p-Fock-space}
  By Proposition~\ref{prop:Hpj-invariance}, the algebra ${\mathfrak{h}_\pj}$ acts on $H^+_\pj$.  We call this the \emph{lowest weight Fock space representation} of ${\mathfrak{h}_\pj}$ and denote it by $\cF_\pj = \cF_\pj(A)$.  Note that this representation is generated by the lowest weight vacuum vector $1 \in H^+_\pj$.
\end{defin}

\begin{prop} \label{prop:p-Fock-space-properties}
  The Fock space $\cF_\pj$ of ${\mathfrak{h}_\pj}$ has the following properties.
  \begin{enumerate}
    \item \label{prop-item:p-Fock-space-subreps} The only submodules of $\cF_\pj$ are those submodules of the form $n \cF_\pj$ for $n \in \Z$.

    \item Let $\Z^- \cong \Z$ (isomorphism of $\Z$-modules) be the representation of $H^-$ such that $H^-_n$ acts as zero for all $n > 0$ and $H^-_0 \cong \Z$ acts by left multiplication.  Then $\cF_\pj$ is isomorphic to the induced module $\Ind^{\mathfrak{h}_\pj}_{H^-} \Z^-$ as an ${\mathfrak{h}_\pj}$-module.

    \item \label{prop-item:p-Stone-von-Neumann} Any representation of ${\mathfrak{h}_\pj}$ generated by a lowest weight vacuum vector is isomorphic to $\cF_\pj$.
  \end{enumerate}
\end{prop}

\begin{proof}
  This is an immediate consequence of Theorem~\ref{theo:Fock-space-properties}, taking $X^+ = H^+_\pj$.
\end{proof}

Note that the lowest weight Fock space $\cF_\pj$ is not a faithful ${\mathfrak{h}_\pj}$-module in general (see Section~\ref{subsec:quasi-Fock-spaces}), in contrast to the case for $\fh$ (see Theorem~\ref{theo:Fock-space-properties}\eqref{theo-item:Fock-space-faithful}).  However, we can define a highest weight Fock space of ${\mathfrak{h}_\pj}$ that is faithful.  Consider the augmentation algebra homomorphism $\epsilon^+ \colon H^+_\pj \to \Z$ uniquely determined by $\epsilon^+(H^+_n \cap H^+_\pj) = 0$ for $n > 0$.  Let $\Z_\epsilon^+$ denote the corresponding $H^+_\pj$-module.  We call the induced module $\cF_\pj^- := \Ind^{\mathfrak{h}_\pj}_{H^+_\pj} \Z_\epsilon^+$ the \emph{highest weight Fock space representation} of ${\mathfrak{h}_\pj}$.  It is generated by the highest weight vacuum vector $1 \in \Z_\epsilon^+$.

\begin{prop} \label{prop:hw-p-FS-faithful}
  The highest weight Fock space representation $\cF_\pj^-$ of ${\mathfrak{h}_\pj}$ is faithful.
\end{prop}

\begin{proof}
  We have ${\mathfrak{h}_\pj} \cong H^+_\pj \otimes H^-$ as $\kk$-modules and so $\cF_\pj^- \cong H^-$ as $\kk$-modules.  The action of ${\mathfrak{h}_\pj}$ on $\cF_\pj^-$ is simply the restriction of the natural action of $\fh$ on $H^-$, which is faithful by (an obvious highest weight analogue of) Theorem~\ref{theo:Fock-space-properties}\eqref{theo-item:Fock-space-faithful}.
\end{proof}

\subsection{Categorification of Fock space}

In this section we prove our main result, the categorification of the Fock space representation of the Heisenberg double.  We continue to fix a dualizing tower of algebras $A$ and to use the notation of~\eqref{eq:tower-notation}.

Recall the direct sums of categories
\[ \ts
  A\md = \bigoplus_{n \in \N} A_n\md,\quad A\pmd = \bigoplus_{n \in \N} A_n\pmd.
\]
For each $M \in A_m\md$, $m \in \N$, define the functor $\Ind_M \colon A\md \to A\md$ by
\[
  \Ind_M(N) = \Ind^{A_{m+n}}_{A_m \otimes A_n} (M \otimes N) \in A_{m+n}\md,\quad N \in A_n\md,\ n \in \N.
\]
For each $P \in A_p\pmd$, $p \in \N$, define the functor $\Res_P \colon A\md \to A\md$ by
\[
  \Res_P(N) = \Hom_{A_p}(P, \Res^{A_n}_{A_{n-p} \otimes A_p} N) \in A_{n-p}\md,\quad N \in A_n\md,\ n \in \N,
\]
where $\Res_P(N)$ is interpreted to be the zero object of $A\md$ if $n-p<0$.

As in the proof of Proposition~\ref{prop:Hpj-invariance}, we see that
\[
  \Ind_P(A\pmd) \subseteq A\pmd, \quad \Res_P(A\pmd) \subseteq A\pmd \quad \text{for all } P \in A\pmd.
\]
Thus we have the induced functors $\Ind_P, \Res_P \colon A\pmd \to A\pmd$ for $P \in A\pmd$.

Since the functors $\Ind_M$ and $\Res_P$ are exact for all $M \in A\md$ and $P \in A\pmd$, they induce endomorphisms $[\Ind_M]$ and $[\Res_P]$ of $\cG(A)$.  Similarly, $\Ind_P$ and $\Res_P$ induce endomorphisms $[\Ind_P]$ and $[\Res_P]$ of $\cG_\pj(A)$ for all $P \in A\pmd$.

\begin{prop} \label{prop:cat-Groth-group}
  Suppose $A$ is a dualizing tower of algebras.
  \begin{enumerate}
    \item \label{theo-item:weak-cat} For all $M,N \in A\md$ and $P \in A\pmd$, we have
        \[
          ([M] \# [P])([N]) = [\Ind_M] \circ [\Res_P] ([N]) = [\Ind_M \circ \Res_P (N)] \in \cG(A).
        \]

    \item \label{theo-item:weak-cat-proj} For all $Q,P,R \in A\pmd$, we have
        \[
          ([Q] \# [P])([R]) = [\Ind_Q] \circ [\Res_P] ([R]) = [\Ind_Q \circ \Res_P (R)] \in \cG_\pj(A).
        \]
  \end{enumerate}
\end{prop}

\begin{proof}
  This follows from the definition of the multiplication in $\cG(A)$ and Lemma~\ref{lem:adjoint-action-explicit}.
\end{proof}

Part~\eqref{theo-item:weak-cat} (resp.\ part~\eqref{theo-item:weak-cat-proj}) of Proposition~\ref{prop:cat-Groth-group}
shows how the action of $\fh$ on $\cF$ (resp. $\fh_\pj$ on $\cF_\pj$) is induced by functors on $\bigoplus_{n\geq0} A_n\md$ (resp. $\bigoplus_{n\geq0}A_n\pmd$).  Typically a \emph{categorification} of a representation consists of isomorphisms of such functors which lift the algebra relations.  As we now describe, this can be done if the tower of algebras is strong.

First note that the algebra structure on $\fh$ is uniquely determined by the fact that $H^\pm$ are subalgebras and by the relation
\begin{equation} \label{eq:fh-key-relation} \ts
  xa = \sum_{(x)} \pR{x_{(1)}}^*(a) x_{(2)},\quad x \in H^-,\ a \in H^+,
\end{equation}
between $H^-$ and $H^+$.  Since the natural action of $\fh$ on $H^+$ is faithful by Theorem~\ref{theo:Fock-space-properties}\eqref{theo-item:Fock-space-faithful}, equation~\eqref{eq:fh-key-relation} is equivalent to the equalities in $\End H^+$
\begin{equation} \label{eq:fh-key-relation-action}
  \pR{x}^* \circ \pL{a} = \nabla \left( \pR{\Delta(x)}^* (a \otimes -)\right),\quad x \in H^-,\ a \in H^+.
\end{equation}

For $Q \in (A_p \otimes A_q)\pmd$, $M \in A_m\md$, and $N \in A_n\md$, define
\[
  \Res_Q(M \otimes N) := \Hom_{A_p \otimes A_q} \left( Q, S_{23} \left( \Res^{A_m}_{A_{m-p} \otimes A_p} M \otimes \Res^{A_n}_{A_{n-q} \otimes A_q} N \right) \right).
\]
For $Q \in A\md^{\otimes 2}$, we let $\Res_Q$ denote the corresponding sum of functors.

\begin{theo} \label{theo:categorification}
  Suppose that $A$ is a strong tower of algebras.  Then we have the following isomorphisms of functors for all $M,N \in A\md$ and $P,Q \in A\pmd$.
  \begin{gather}
    \Ind_M \circ \Ind_N \cong \Ind_{\nabla(M \otimes N)}, \label{cat-eq:ind} \\
    \Res_P \circ \Res_Q \cong \Res_{\nabla(P \otimes Q)}, \label{cat-eq:res} \\ \ts
    \Res_P \circ \Ind_M \cong \nabla \Res_{\Psi^{\otimes 2} \Delta \Psi^{-1} (P)}(M \otimes -). \label{cat-eq:cross}
  \end{gather}
  In particular, if $A$ is dualizing, then the above yields a categorification of the lowest weight Fock space representations of $\fh(A)$ and $\fh_\pj(A)$.
\end{theo}

The isomorphisms~\eqref{cat-eq:ind} and~\eqref{cat-eq:res} categorify the multiplication in $\cG(A)$ and $\cK(A)$ respectively.  If $A$ is dualizing, then, in light of Proposition~\ref{prop:dualizing-conjugation-isom}, the isomorphism~\eqref{cat-eq:cross} categorifies the relation~\eqref{eq:fh-key-relation-action}.   Thus, Theorem~\ref{theo:categorification} provides a categorification of the lowest weight Fock space representation $\cF(A)$ of $\fh(A)$.  If we restrict the induction and restriction functors to $A\pmd$ and require $M,N \in A\pmd$ in the statement of the theorem, then we obtain a categorification of the lowest weight Fock space representation $\cF_\pj(A)$ of $\fh_\pj(A)$.  Note that the categorification in Theorem~\ref{theo:categorification} does not rely on a particular presentation of the Heisenberg double $\fh(A)$ (see Remark~\ref{rem:Hecke-presentations}), in contrast to many other categorification statements appearing in the literature.

\begin{proof}[Proof of Theorem~\ref{theo:categorification}]
  Suppose $M \in A_m\md$, $N \in A_n\md$, $P \in A_p\pmd$, $Q \in A_q\pmd$, and $L \in A_\ell\md$.  Then we have
  \begin{align*}
    \Ind_M \circ \Ind_N (L) &= \Ind^{A_{m+n+\ell}}_{A_m \otimes A_{n + \ell}} \left( M \otimes \Ind^{A_{n + \ell}}_{A_n \otimes A_\ell} (N \otimes L) \right) \\
    &\cong \Ind^{A_{m+n+\ell}}_{A_m \otimes A_{n + \ell}} \Ind^{A_m \otimes A_{n + \ell}}_{A_m \otimes A_n \otimes A_\ell} (M \otimes N \otimes L) \\
    &\cong \Ind^{A_{m+n+\ell}}_{A_m \otimes A_n \otimes A_\ell} (M \otimes N \otimes L) \\
    &\cong \Ind^{A_{m+n+\ell}}_{A_{m+n} \otimes A_\ell} \Ind^{A_{m+n} \otimes A_\ell}_{A_m \otimes A_n \otimes A_\ell} (M \otimes N \otimes L) \\
    &\cong \Ind^{A_{m+n+\ell}}_{A_{m+n} \otimes A_\ell} \left( \Ind^{A_{m+n}}_{A_m \otimes A_n} (M \otimes N) \otimes L \right) \\
    &\cong \Ind_{\nabla(M \otimes N)} L.
  \end{align*}
  Since each of the above isomorphisms is natural in $L$, this proves~\eqref{cat-eq:ind}.

  Similarly, we have
  \begin{align*}
    \Res_P \circ \Res_Q (L) &= \Hom_{A_p} (P, \Res^{A_{\ell-q}}_{A_{\ell-q-p} \otimes A_p} \Hom_{A_q} (Q, \Res^{A_\ell}_{A_{\ell-q} \otimes A_q} L)) \\
    &\cong \Hom_{A_p} (P, \Hom_{A_q} (Q, \Res^{A_\ell}_{A_{\ell-p-q} \otimes A_p \otimes A_q} L)) \\
    &\cong \Hom_{A_p \otimes A_q} (P \otimes Q, \Res^{A_\ell}_{A_{\ell-p-q} \otimes A_p \otimes A_q} L) \\
    &\cong \Hom_{A_p \otimes A_q} (P \otimes Q, \Res^{A_{\ell-p-q} \otimes A_{p+q}}_{A_{\ell-p-q} \otimes A_p \otimes A_q} \Res^{A_\ell}_{A_{\ell-p-q} \otimes A_{p+q}} L) \\
    &\cong \Hom_{A_{p+q}} (\Ind^{A_{p+q}}_{A_p \otimes A_q} (P \otimes Q),\Res^{A_\ell}_{A_{\ell-p-q} \otimes A_{p+q}} L) \\
    &\cong \Res_{\nabla(P \otimes Q)} L.
  \end{align*}
  Since each of the above isomorphisms is natural in $L$, this proves~\eqref{cat-eq:res}.

  Finally, we have
  \begin{align*}
    \Res_P &\circ \Ind_M (L) = \Hom_{A_p} (P, \Res^{A_{m+\ell}}_{A_{m+\ell-p} \otimes A_p} \Ind^{A_{m+\ell}}_{A_m \otimes A_\ell} (M \otimes L)) \\
    &\cong \ts \Hom_{A_p} (P, \bigoplus_{s+t=p} \Ind^{A_{m+\ell-p} \otimes A_p}_{A_{m-s} \otimes A_s \otimes A_{\ell-t} \otimes A_t} \Res^{A_m \otimes A_\ell}_{A_{m-s} \otimes A_s \otimes A_{\ell-t} \otimes A_t} (M \otimes L)) \\
    &\cong \ts \Hom_{A_p} (P, \bigoplus_{s+t=p} \Ind^{A_{m+\ell-p} \otimes A_p}_{A_{m-s} \otimes A_{\ell-t} \otimes A_p} \Ind^{A_{m-s} \otimes A_{\ell-t} \otimes A_p}_{A_{m-s} \otimes A_s \otimes A_{\ell-t} \otimes A_t} \Res^{A_m \otimes A_\ell}_{A_{m-s} \otimes A_s \otimes A_{\ell-t} \otimes A_t} (M \otimes L)) \\
    &\cong \ts \bigoplus_{s+t=p} \Ind^{A_{m+\ell-p}}_{A_{m-s} \otimes A_{\ell-t}} \Hom_{A_p} (P, \Ind^{A_{m-s} \otimes A_{\ell-t} \otimes A_p}_{A_{m-s} \otimes A_s \otimes A_{\ell-t} \otimes A_t} \Res^{A_m \otimes A_\ell}_{A_{m-s} \otimes A_s \otimes A_{\ell-t} \otimes A_t} (M \otimes L)) \\
    &\cong \ts \bigoplus_{s+t=p} \Ind^{A_{m+\ell-p}}_{A_{m-s} \otimes A_{\ell-t}} \Hom_{A_s \otimes A_t} (\Psi^{\otimes 2} \Res^{A_p}_{A_s \otimes A_t} \Psi^{-1} (P), \Res^{A_m \otimes A_\ell}_{A_{m-s} \otimes A_s \otimes A_{\ell-t} \otimes A_t} (M \otimes L)) \\
    &\cong \nabla \Res_{\Psi^{\otimes 2} \Delta \Psi^{-1} (P)} (M \otimes L),
  \end{align*}
  where the first isomorphism follows from~\eqref{eq:strong-isom}.  Since all of the above isomorphisms are natural in $L$,
  this proves~\eqref{cat-eq:cross}.

  The final assertion of the theorem follows from Proposition~\ref{prop:dualizing-conjugation-isom} as explained in the paragraph following the statement of the theorem.
\end{proof}

%
\section{Hecke-like towers of algebras}
%

In the remainder of the paper we will be studying some well known examples of towers of algebras.  These examples are all quotients of groups algebras of braid groups by quadratic relations.  In this subsection, we prove that all such towers of algebras are strong and dualizing.  In this section, $\F$ is an algebraically closed field unless otherwise specified.

\begin{defin}[Hecke-like algebras and towers] \label{def:Hecke-like}
  We say that the $\F$-algebra $B$ is \emph{Hecke-like} (of degree $n$) if it is generated by elements $T_1,\dotsc,T_{n-1}$, subject to the relations
  \begin{gather*}
    T_i T_j = T_j T_i,\quad |i-j|>1, \\
    T_i T_{i+1} T_i = T_{i+1} T_i T_{i+1},\quad i=1,\dotsc,n-2, \\
    T_i^2 = c T_i + d,\quad i=1,\dotsc,n-1,
  \end{gather*}
  for some $c,d \in \F$ (independent of $i$).  In other words, $B$ is Hecke-like if it is a quotient of the group algebra of the braid group by quadratic relations (the last set of relations above).

  If, for $n \in \N$, the algebra $A_n$ is a Hecke-like algebra of degree $n$, and the constants $c,d$ above are independent of $n$, then we can define an external multiplication on $A = \bigoplus_{n \in \N} A_n$ by
  \begin{equation} \label{def-eq:ext-mult}
    \rho_{m,n} \colon A_m \otimes A_n \to A_{m+n},\quad T_i \otimes 1 \mapsto T_i,\ 1 \otimes T_i \mapsto T_{m+i}.
  \end{equation}
  Axioms (TA1) and (TA2) follow immediately.  Furthermore, it follows from Lemma~\ref{lem:Hecke-like-properties} below that, as a left $(A_m \otimes A_n)$-module, we have
  \[ \ts
    A_{m+n} = \bigoplus_{w \in X_{m,n}} (A_m \otimes A_n) T_w,
  \]
  where $X_{m,n}$ is a set of minimal length representatives of the cosets $(S_m \times S_n) \backslash S_{m+n}$.  Thus $A_{m+n}$ is a projective left $(A_m \otimes A_n)$-module.  Similarly, it is also a projective right $(A_m \otimes A_n)$-module and so axiom (TA3) is satisfied.   Finally (TA4) is satisfied since $\F$ is algebraically closed.  We call the resulting tower of algebras a \emph{Hecke-like tower of algebras}.
\end{defin}

\begin{lem} \label{lem:Hecke-like-properties}
  Suppose that $B$ is a Hecke-like algebra of degree $n$.
  \begin{enumerate}
    \item \label{lem-item:Hecke-like-basis} If, for $w \in S_n$, we define $T_w = T_{i_1} \dotsm T_{i_r}$, where $s_{i_1} \dotsm s_{i_r}$ is a reduced decomposition of $w$ (these elements are well defined by the braid relations), then $\{T_w\ |\ w\in S_n\}$ is a basis of $B$.  In particular, the dimension of $B$ is $n!$.
    \item \label{lem-item:Hecke-like-mult} We have
      \[
        T_{w_1} T_{w_2} = T_{w_1 w_2} \quad \text{for all } w_1,w_2 \in S_n \text{ such that } \ell(w_1 w_2) = \ell(w_1) + \ell(w_2),
      \]
      where $\ell$ is the usual length function on $S_n$.

    \item \label{lem-item:Hecke-like-Nakayama} The algebra $B$ is a Frobenius algebra with trace map $\lambda \colon B \to \F$ given by $\lambda(T_w) = \delta_{w,w_0}$, where $w_0$ is the longest element of $S_n$.  The corresponding Nakayama automorphism is the map $\psi_n \colon B \to B$ given by $\psi_n(T_i) = T_{n-i}$.
    \end{enumerate}
\end{lem}

\begin{proof}
  The proof of parts~\eqref{lem-item:Hecke-like-basis} and~\eqref{lem-item:Hecke-like-mult} is analogous to the proof for the usual Hecke algebra of type $A$ and is left to the reader.  It remains to prove part~\eqref{lem-item:Hecke-like-Nakayama}.

  Suppose $B$ is a Hecke-like algebra of degree $n$.  To show that $B$ is a Frobenius algebra with trace map $\lambda$, it suffices to show that $\ker \lambda$ contains no nonzero left ideals.  Let $I$ be a nonzero ideal of $B$.  Then choose a nonzero element $b = \sum_{w \in S_n} a_w T_w \in I$ and let $\tau$ be a maximal length element of the set $\{w \in S_n\ |\ a_w \ne 0\}$.  Then we have $\lambda(T_{w_0 \tau^{-1}} b) = a_\tau \ne 0$.  Thus $I$ is not contained in $\ker \lambda$.

  To show that $\psi_n$ is the Nakayama automorphism, it suffices to show that $\lambda (T_w T_i) = \lambda(T_{n-i} T_w)$ for all $i \in \{1,\dotsc,n-1\}$ and $w \in S_n$.  We break the proof into four cases.

  \emph{Case 1:} $\ell(w) \le \ell(w_0) - 2$.  In this case we clearly have $\lambda(T_w T_i) = 0 = \lambda(T_{n-i} T_w)$.

  \emph{Case 2:} $w = w_0$.  Then we can write $w = w_0 = \tau s_i$ for some $\tau \in S_n$ with $\ell(\tau) = \ell(w_0)-1$.  Then $\lambda(T_w T_i) = \lambda(cT_{w_0} + dT_\tau) = c$.  Now, since $w_0 s_i = s_{n-i} w_0$, we have $w = w_0 = s_{n-i} \tau$.  Thus $\lambda(T_{n-i} T_w) = \lambda(cT_{w_0} + d T_\tau) = c = \lambda(T_w T_i)$.

  \emph{Case 3:} $w s_i = w_0$.  Since, as noted above, we have $w_0 s_i = s_{n-i} w_0$, it follows that $s_{n-i} w = w_0$ and so $\lambda(T_w T_i) = \lambda(T_{w_0}) = \lambda(T_{n-i} T_w)$.

  \emph{Case 4:} $\ell(w) = \ell(w_0) - 1$, but $w s_i \ne w_0$.  Then we have $w = \tau s_i$ for some $\tau \in S_n$ with $\ell(\tau) = \ell(w)-1$.  Thus $\lambda(T_w T_i) = \lambda(cT_w + d T_\tau) = 0$.  Using again the equality $w_0 s_i = s_{n-i} w_0$, we have $s_{n-i} w \ne w_0$.  Then an analogous argument shows that $\lambda(T_{n-i} T_w)=0$.
\end{proof}

When considering a Hecke-like tower of algebras, we will always assume that the conjugation $\Psi$ is given by the Nakayama automorphisms $\psi_n$ described above (see Proposition~\ref{lem:Frobenius-conjugate}).

\begin{prop} \label{prop:Hecke-like}
  All Hecke-like towers of algebras are strong.
\end{prop}

\begin{proof}
  Suppose $A = \bigoplus_{n \in \N} A_n$ is a Hecke-like tower of algebras.  We formulate the isomorphism~\eqref{eq:strong-isom} in terms of bimodules.  Fix $n,m,k,\ell$ such that $n+m=k+\ell$ and set $N=n+m$. Let $_{(k,\ell)}(A_N)_{(n,m)}$ denote $A_N$, thought of as an $(A_k\otimes A_\ell,A_n\otimes A_m)$-bimodule in the natural way.  Then we have
  \[
    \Res^{A_N}_{A_k \otimes A_\ell} \Ind^{A_N}_{A_n \otimes A_m} \cong {_{(k,\ell)}(A_N)_{(n,m)}} \otimes -.
  \]
  On the other hand, for each $r$ satisfying $k-m=n-\ell \le r \le \min \{n,k\}$, we have
  \begin{gather*}
    \Ind^{A_k \otimes A_\ell}_{A_r \otimes A_{n-r} \otimes A_{k-r} \otimes A_{\ell+r-n}} \Res^{A_n \otimes A_m}_{A_r \otimes A_{n-r} \otimes A_{k-r} \otimes A_{\ell+r-n}} \cong B_r \otimes -,\quad \text{ where} \\
    B_r=(A_k\otimes A_\ell)\otimes _{A_r\otimes A_{n-r} \otimes A_{k-r}\otimes A_{\ell+r-n}}(A_n\otimes A_m),
  \end{gather*}
  and where we view $A_k \otimes A_\ell$ as a right $(A_r \otimes A_{n-r} \otimes A_{k-r} \otimes A_{\ell + r - n})$-module via the map $a_1 \otimes a_2 \otimes a_3 \otimes a_4 \mapsto a_1 a_3 \otimes a_2 a_4$.  (This corresponds to the functor $S_{23}$ appearing in~\eqref{eq:strong-isom}.)  Therefore, in order to prove the isomorphism~\eqref{eq:strong-isom}, it suffices to prove that we have an isomorphism of bimodules
  \[ \ts
    _{(k,\ell)}(A_N)_{(n,m)} \cong \bigoplus_{r=n-\ell}^{\min\{n,k\}} B_r.
  \]

  Now, we have one double coset in $S_k\times S_\ell \setminus S_N / S_n \times S_m$ for each $r$ satisfying $k-m=n-\ell \le r \le \min \{n,k\}$ (see, for example, \cite[Appendix~3, p.~170]{Zel81}).  Precisely, the double coset $C_r$ corresponding to $r$ consists of the permutations $w \in S_n$ satisfying
  \begin{gather*}
    |w(\{1,\dotsc,n\}) \cap \{1,2,\dotsc,k\}| = r,\\
    |w(\{n+1,\dotsc,N\}) \cap \{1,2,\dotsc,k\}| = k-r, \\
    |w(\{1,\dotsc,n\}) \cap \{k+1,\dotsc,N\}| = n-r,\\
    |w(\{n+1,\dotsc,N\}) \cap \{k+1,\dotsc,N\}| = \ell-n+r = m-k+r.
  \end{gather*}
  Thus the cardinality of the double coset $C_r$ is
  \[
    |C_r| = m!n! \binom{k}{r} \binom{\ell}{n-r}.
  \]
  The permutation $w_r \in S_N$ given by
  \[
    w_r(i) =
    \begin{cases}
      i & \text{if } 1 \le i \le r, \\
      i - r + k & \text{if } r < i \le n, \\
      i - n + r & \text{if } n < i \le n+k-r, \\
      i & \text{if } n+k-r < i \le N,
    \end{cases}
  \]
  is a minimal length representative of $C_r$.  It then follows from Lemma~\ref{lem:Hecke-like-properties}\eqref{lem-item:Hecke-like-mult} that
  \begin{gather*}
    T_{w_r} T_i = T_i T_{w_r} \quad \text{if } 1 \le i < r \text{ or } n+k-r < i < N, \\
    T_{w_r} T_i = T_{i-r+k} T_{w_r} \quad \text{if } r < i < n, \\
    T_{w_r} T_i = T_{i-n+r} T_{w_r} \quad \text{if } n < i < n+k-r.
  \end{gather*}
  Thus,
  \begin{equation} \label{eq:Twr-commutation}
    T_{w_r} (a_1 \otimes a_2 \otimes a_3 \otimes a_4) = (a_1 \otimes a_3 \otimes a_2 \otimes a_4) T_{w_r}
  \end{equation}
  for all $a_1 \in A_r$, $a_2 \in A_{n-r}$, $a_3 \in A_{k-r}$, $a_4 \in A_{\ell+r-n}$.

  Define $B_r' \subseteq {_{(k,\ell)}(A_N)_{(n,m)}}$ to be the sub-bimodule generated by $T_{w_r}$.  It follows from Lemma~\ref{lem:Hecke-like-properties} that ${_{(k,\ell)}(A_N)_{(n,m)}}=\bigoplus_{r=n-\ell}^{\min\{n,k\}} B_r'$ and that $\dim_\F B_r' = |C_r|$.  It also follows that the dimension of $A_k \otimes A_\ell$ as a right module over $A_r \otimes A_{n-r} \otimes A_{k-r} \otimes A_{\ell-n+r}$ is $k! \ell! / r! (n-r)! (k-r)! (\ell - n + r)!$ and that the dimension of $A_m \otimes A_n$ as a left module over $A_r \otimes A_{n-r} \otimes A_{k-r} \otimes A_{\ell+r-n}$ is $m! n! / r! (n-r)! (k-r)! (\ell - n+ r)!$.  Therefore, $\dim_\F B_r = |C_r| = \dim_\F B_r'$.  Now consider the $(A_k \otimes A_\ell, A_n \otimes A_m)$-bimodule map $B_r \to B_r'$ uniquely determined by
  \[
    1_{A_k \otimes A_\ell} \otimes 1_{A_n \otimes A_m} \mapsto T_{w_r},
  \]
  which is well defined by~\eqref{eq:Twr-commutation}.
  This map is surjective, and thus is an isomorphism by dimension considerations.
\end{proof}

\begin{lem} \label{lem:Hecke-like-coprod-conjugation}
  If $A$ is a Hecke-like tower of algebras, then we have an isomorphism of functors $\Psi^{\otimes 2} \Delta \Psi^{-1} \cong S_{12} \Delta$ on $A\md$ (hence also on $A\pmd$).
\end{lem}

\begin{proof}
  It suffices to prove that, for $m, n \in \N$, we have an isomorphism of functors
  \[
    (\Psi_m \otimes \Psi_n) \circ \Res^{A_{m+n}}_{A_m \otimes A_n} \circ \Psi_{m+n}^{-1} \cong S_{12} \circ \Res^{A_{m+n}}_{A_n \otimes A_m}.
  \]
  Describing each functor above as tensoring on the left by the appropriate bimodule, it suffices to prove that we have an isomorphism of bimodules
  \begin{equation} \label{eq:bimodule-isom-comult}
    (A_m^\psi \otimes A_n^\psi) \otimes_{A_m \otimes A_n} A_{m+n}^\psi \cong S \otimes_{A_n \otimes A_m} A_{m+n},
  \end{equation}
  where $A_k^\psi$, $k \in \N$, denotes $A_k$, considered as an $(A_k,A_k)$-bimodule with the right action twisted by $\psi_k$, and where $S$ is $A_n \otimes A_m$ considered as an $(A_m \otimes A_n, A_n \otimes A_m)$-module via the obvious right multiplication and with left action given by $(a_1 \otimes a_2, s) \mapsto (a_2 \otimes a_1) s$ for $s \in S$, $a_1 \in A_m$, $a_2 \in A_n$.

  For $k \in \N$, let $1_k$ denote the identity element of $A_k$ and $1_k^\psi$ denote this same element considered as an element of $A_k^\psi$.   It is straightforward to show that the map between the bimodules in~\eqref{eq:bimodule-isom-comult} given by
  \[
    (1_m^\psi \otimes 1_n^\psi) \otimes 1_{m+n}^\psi \mapsto (1_n \otimes 1_m) \otimes 1_{m+n}.
  \]
  (and extended by linearity) is a well defined isomorphism.
  \details{
  For $i \in \{1,\dotsc,m\}$ and $j \in \{1,\dotsc,n\}$, in $S \otimes_{A_n \otimes A_m} A_{m+n}$ we have
  \begin{gather*}
    (T_i \otimes 1) \left( (1_n \otimes 1_m) \otimes 1_{m+n} \right) = (1_n \otimes T_i) \otimes 1_{m+n} = \left( (1_n \otimes 1_m) \otimes 1_{m+n} \right) T_{n+i}, \\
    (1 \otimes T_j) \left( (1_n \otimes 1_m) \otimes 1_{m+n} \right) = (T_j \otimes 1_m) \otimes 1_{m+n} = \left( (1_n \otimes 1_m) \otimes 1_{m+n} \right) T_j,
  \end{gather*}
  and in $(A_m^\psi \otimes A_n^\psi) \otimes_{A_m \otimes A_n} A_{m+n}^\psi$ we have
  \begin{gather*}
    (T_i \otimes 1) \left( (1_m^\psi \otimes 1_n^\psi) \otimes 1^\psi_{m+n} \right) = \left( (1_m^\psi \otimes 1_n^\psi) (T_{m-i} \otimes 1) \right) \otimes 1_{m+n}^\psi = \left( (1_m^\psi \otimes 1_n^\psi) \otimes 1_{m+n}^\psi \right) T_{n+i}, \\
    (1 \otimes T_j) \left( (1_m^\psi \otimes 1_n^\psi) \otimes 1^\psi_{m+n} \right) = \left( (1_m^\psi \otimes 1_n^\psi) (1 \otimes T_{n-j}) \right) \otimes 1_{m+n}^\psi = \left( (1_m^\psi \otimes 1_n^\psi) \otimes 1_{m+n}^\psi \right) T_j.
  \end{gather*}
  Thus we have a well defined map between the bimodules in~\eqref{eq:bimodule-isom-comult} given by
  \[
    (1_m^\psi \otimes 1_n^\psi) \otimes 1_{m+n}^\psi \mapsto (1_n \otimes 1_m) \otimes 1_{m+n}.
  \]
  This map is clearly surjective.  Since we can construct the inverse map analogously, it is an isomorphism.
  }
\end{proof}

Suppose $A$ is a Hecke-like tower of algebras.  If $d \ne 0$ in Definition~\ref{def:Hecke-like}, then $T_i (T_i-c)/d = (T_i - c)T_i/d = 1$ and so $T_i$ is invertible for all $i$.  It follows that $T_w$ is invertible for all $w \in S_n$.  On the other hand, if $d=0$, then $T_i(T_i-c)=0$ and so $T_i$ is a zero divisor, hence not invertible.  Therefore, $d \ne 0$ if and only if $T_w$ is invertible for all $w$.

\begin{lem} \label{lem:Hecke-like-d-nonzero}
  If $A$ is a Hecke-like tower of algebras with $d \ne 0$ (equivalently, such that $T_i$ is invertible for all $i$), then we have isomorphisms of functors $\nabla \cong \nabla S_{12}$ and $\Delta \cong S_{12} \Delta$ on $A\md$ (hence also on $A\pmd$).  In particular, $A$ is dualizing.
\end{lem}

\begin{proof}
  Let $m, n \in \N$ and define $w \in S_{m+n}$ by
  \[
    w(i) =
    \begin{cases}
      m+i & \text{if } 1 \le i \le n, \\
      i-n & \text{if } n < i \le m+n.
    \end{cases}
  \]
  Then $T_w$ is invertible and, by Lemma~\ref{lem:Hecke-like-properties}\eqref{lem-item:Hecke-like-mult}, we have $T_w T_i = T_{w(i)} T_w$ for all $i=1,\dotsc,m-1,m+1,\dotsc,m+n-1$.  Now, let $S$ be $A_m \otimes A_n$ considered as an $(A_n \otimes A_m, A_m \otimes A_n)$-module via the obvious right multiplication and with left action given by $(a_1 \otimes a_2, s) \mapsto (a_2 \otimes a_1) s$ for $s \in S$, $a_1 \in A_n$, $a_2 \in A_m$.  Thus, we have an isomorphism of functors $S_{12} \cong S \otimes -$.  It is straightforward to verify that the map
  \[
    A_{m+n} \to A_{m+n} \otimes_{A_n \otimes A_m} S,\quad a \mapsto aT_w \otimes (1 \otimes 1),
  \]
  is an isomorphism of $(A_{m+m}, A_m \otimes A_n)$-bimodules.  It follows that $\nabla \cong \nabla S_{12}$.  The proof that $\Delta \cong S_{12} \Delta$ is analogous.  The final statement of the lemma then follows from Lemma~\ref{lem:Hecke-like-coprod-conjugation} and Propositions~\ref{prop:dualizing-conjugation-isom} and~\ref{prop:Hecke-like}.
\end{proof}

\begin{cor} \label{cor:Hecke-like-strong-dualizing}
  All Hecke-like towers of algebras are strong and dualizing.
\end{cor}

\begin{proof}
  It follows immediately from Lemma~\ref{lem:Hecke-like-coprod-conjugation} and Propositions~\ref{prop:dualizing-conjugation-isom} and~\ref{prop:Hecke-like} that $A$ is strong and that it is dualizing if and only if $\cK(A)$ is cocommutative.  By~\cite[Prop.~6.1]{BLL12}, $A$ is isomorphic to either the tower of nilcoxeter algebras, the tower of Hecke algebras at a generic parameter, the tower of Hecke algebras at a root of unity, or the tower of 0-Hecke algebras.  (While the statement of~\cite[Prop.~6.1]{BLL12} is over $\C$, the proof is valid over an arbitrary algebraically closed field.)  For the tower of Hecke algebras at a generic parameter or a root of unity, we have $d \ne 0$ in Definition~\ref{def:Hecke-like} and so $\cK(A)$ is cocommutative by Lemma~\ref{lem:Hecke-like-coprod-conjugation}.   For the other two towers, we will see in Sections~\ref{sec:Weyl} and~\ref{sec:qsym-nsym} that $\cK(A)$ is cocommutative.
\end{proof}

\begin{rem} \label{rem:Hecke-like-general-field}
  In this section, the assumption that $\F$ is algebraically closed was only used to conclude that axiom (TA4) of Definition~\ref{def:tower} is satisfied and in the proof of Corollary~\ref{cor:Hecke-like-strong-dualizing}.  If $\F$ is not algebraically closed but the bilinear form~\eqref{eq:tower-pairing} is still a perfect pairing, then all the results of this section remain true except that one must replace Corollary~\ref{cor:Hecke-like-strong-dualizing} by the statement that the tower of nilcoxeter algebras, the tower of Hecke algebras at a generic parameter, the tower of Hecke algebras at a root of unity, and the tower of 0-Hecke algebras are all strong and dualizing.
\end{rem}

%
\section{Nilcoxeter algebras} \label{sec:Weyl}
%

In this section we specialize the constructions of Sections~\ref{sec:general-construction} and~\ref{sec:dual-pairs-from-towers} to the tower of nilcoxeter algebras of type $A$.  We will see that we recover Khovanov's categorification of the polynomial representation of the Weyl algebra (see \cite{Kho01}).  We let $\F$ be an arbitrary field.

The \emph{nilcoxeter algebra} $N_n$ is the Hecke-like algebra of Definition~\ref{def:Hecke-like} with $c=d=0$.  The representation theory of $N_n$ is straightforward.  (We refer the reader to~\cite{Kho01} for proofs of the facts stated here.)  Up to isomorphism, there is one simple module $L_n$, which is one dimensional, and on which the generators all act by zero.  The projective cover of $L_n$ is $P_n=N_n$, considered as an $N_n$-module by left multiplication.  We have isomorphisms of Hopf algebras
\begin{gather*}
  \cK(N) \cong \Z[x],\quad [P_n] \mapsto x^n, \\
  \cG(N) \cong \Z[x,x^2/2!, x^3/3!,\dotsc] \subseteq \Q[x],\quad [L_n] \mapsto x^n/n!.
\end{gather*}
In both cases, the coproduct is given by $\Delta(x) = x \otimes 1 + 1 \otimes x$.  We also have
\[
  \cG_\pj(N) \cong \Z[x],
\]
and the Cartan map $\cK(N) \to \cG(N)$ of Definition~\ref{def:G-proj} corresponds to the natural inclusion $\Z[x] \hookrightarrow \Z[x,x^2/2!,x^3/3!,\dotsc]$.

The inner product satisfies $\langle x^m,\frac{x^n}{n!} \rangle = \langle [P_m] , [L_n] \rangle =\delta_{mn}$.  Therefore $x^*\left(\frac{x^m}{m!}\right) = \frac{x^{m-1}}{(m-1)!}$, i.e.\ $x^*=\partial_x$ corresponds to partial derivation by $x$.  Therefore the algebra $\fh$ in this setting is the subalgebra of $\End \Z[x,x^2/2,x^3/3!,\dotsc]$ generated by $x,x^2/2!,x^3/3!,\dotsc$ and $\partial_x$.  In addition, ${\mathfrak{h}_\pj}$ is the algebra generated by $x$ and $\partial_x$, with relation $[\partial_x, x] = 1$.  The Fock space $\cF_\pj$ is the representation of ${\mathfrak{h}_\pj}$ given by its natural action on $\Z[x]$.  It follows from the above that $\Q \otimes_\Z \fh \cong \Q \otimes_\Z {\mathfrak{h}_\pj}$ is the rank one Weyl algebra.

By Remark~\ref{rem:Hecke-like-general-field}, the tower $N$ is strong and dualizing.  Thus, Theorem~\ref{theo:categorification} provides a categorification of the polynomial representation of the Weyl algebra.  In fact,~\eqref{cat-eq:cross} specializes to the main result of \cite{Kho01} if one takes $M$ and $P$ to be the trivial $N_1$-modules.  Indeed, with these choices we have $\Psi^{\otimes 2} \Delta \Psi^{-1}(P) = (\F_0 \otimes \F_1) \oplus (\F_1 \otimes \F_0)$, where $\F_i$ denotes the trivial $N_i$-module for $i=0,1$.  Then~\eqref{cat-eq:cross} becomes
\begin{align*}
  \Res^{N_{n+1}}_{N_n} \circ \Ind^{N_{n+1}}_{N_n} &\cong \left(\Ind^{N_n}_{N_{n-1}} \circ \Res_{(\F_0 \otimes \F_1)}(\F_1 \otimes -)\right) \oplus \left(\Res_{(\F_1 \otimes \F_0)} (M \otimes -)\right) \\
  &\cong \left(\Ind^{N_n}_{N_{n-1}} \circ \Res^{N_n}_{N_{n-1}}\right) \oplus \id,
\end{align*}
which is the categorification of the relation $\partial_x x = x \partial_x + 1$ appearing in~\cite[(13)]{Kho01}.  (Note that while~\cite{Kho01} works over the field $\Q$, the arguments go through over more general $\F$.)

%
\section{Hecke algebras at generic parameters} \label{sec:sym}
%

In this section we specialize the constructions of Sections~\ref{sec:general-construction} and~\ref{sec:dual-pairs-from-towers} to the tower of algebras corresponding to the Hecke algebras of type $A$ at a generic parameter.  The results of this section also apply to the group algebra of the symmetric group (the case when $q=1$).

\subsection{The Hecke algebra and symmetric functions} \label{subsec:sym-functions}

Let $A_n$ be the Hecke algebra at a generic value of $q$.  More precisely, assume $q \in \C^\times$ is not a nontrivial root of unity and let $A_n$ be the unital $\C$-algebra of Definition~\ref{def:Hecke-like} with $c=q-1$ and $d=q$.  By convention, we set $A_0 = A_1 = \C$.  Then $A = \bigoplus_{n \in \N} A_n$ is a Hecke-like tower of algebras.  It is well known that a complete set of irreducible $A_n$-modules is given by $\{S_\lambda\ |\ \lambda \in \partition(n)\}$, where $S_\lambda$ is the Specht module corresponding to the partition $\lambda$ (see~\cite[\S6]{DJ86}).  Since the $A_n$ are semisimple, we have $\cK(A) = \cG(A)$.  In fact, both are isomorphic (as Hopf algebras) to $\Sy$, the algebra of symmetric functions in countably many variables $x_1,x_2,\dotsc$ over $\Z$.  This isomorphism is given by the map sending $[S_\lambda]$ to $s_\lambda$, the Schur function corresponding to the partition $\lambda$ (see, for example, \cite{Zel81}).  Recall that $\Sy$ is a graded connected Hopf algebra:
\[ \ts
  \Sy = \bigoplus_{n \ge 0} \Sy_n,
\]
where $\Sy_n$ is the $\Z$-submodule of $\Sy$ consisting of homogeneous polynomials of degree $n$.  We adopt the convention that $\Sy_n = 0$ for $n < 0$.  The inner product~\eqref{eq:tower-pairing} corresponds to the usual inner product on $\Sy$ under which the Schur functions are self-dual.  Furthermore, the monomial and homogeneous symmetric functions are dual to each other:
\[
  \langle m_\lambda,h_\mu \rangle = \delta_{\lambda,\mu},\quad \lambda, \mu \in \partition.
\]
Under this inner product, $\Sy$ is self-dual as a Hopf algebra.  In other words, $(\Sy, \Sy)$ is a dual pair of Hopf algebras.

\subsection{The Heisenberg algebra} \label{subsec:classical-h}

Applying the construction of Section~\ref{subsec:h-definition} to the dual pair $(\Sy,\Sy)$, we obtain the \emph{Heisenberg algebra} $\fh = \fh(\Sy,\Sy)$.  We obtain a minimal presentation of $\fh$ by considering two collections of polynomial generators for $\Sy$ (one for $\Sy$ viewed as $H^+$ and one for $\Sy$ viewed as $H^-$).  In particular, if we choose the power sum symmetric functions $p_n$, $n \in \N_+$, in both cases, we get the usual presentation of the Heisenberg algebra:
\[
  [p_n, p_k] = 0,\quad [p_n^*, p_k^*] = 0,\quad [p_n^*,p_k] = n \delta_{n,k},\quad n,k \in \N_+.
\]
However, $\{p_n^*, p_n\ |\ n \in \N_+\}$ is only a generating set for $\fh \otimes_\Z \Q$ since the power sum symmetric functions only generate the ring of symmetric functions over $\Q$.

One the other hand, if we choose the elementary symmetric functions $e_n$, $n \in \N_+$, and the complete symmetric functions $h_n$, $n \in \N_+$, we have the following relations:
\begin{equation} \label{eq:heisenberg-eh-relations}
  [e_n, e_k] = 0,\quad [h_n^*,h_k^*] = 0,\quad [h_n^*, e_k] = e_{k-1}h_{n-1}^*,\quad n,k \in \N_+.
\end{equation}
(We adopt the convention that $e_0=e_0^*=h_0=h_0^*=1$ and $e_n=e_n^*=h_n=h_n^*=0$ for $n<0$.)  This gives a presentation of $\fh$ (one does not need to tensor with $\Q$) and is the one used in the categorification of $\fh$ given in \cite{Kho10,LS13} (for an overview, see \cite{LS12}).

Other choices of polynomial generators result in different presentations.  For the sake of completeness we record the other nontrivial relations:
\begin{equation} \label{eq:heisenberg-other-relations} \ts
  [e_n^*, h_k] = h_{k-1}e_{n-1}^*,\ [h_n^*, h_k] = \sum_{i\geq 1}h_{k-i}h_{n-i}^*,\ [e_n^*, e_k] = \sum_{i\geq 1}e_{k-i}e_{n-i}^*,\ n,k \in \N_+.
\end{equation}
To prove these relations, we use the fact (see, for example,~\cite[Prop.~3.6]{Z03}) that, for $k,n \in \N$, we have
\begin{gather*}
  h_k^*(h_n) = h_{n-k},\quad h_k^*(e_n) = (\delta_{k0} + \delta_{k1})e_{n-k},\quad h_k^*(p_n) = \delta_{kn} + \delta_{k0} p_n, \\
  e_k^*(h_n) = (\delta_{k0} + \delta_{k1})h_{n-k},\quad e_k^*(e_n)=e_{n-k},\ e_k^*(p_n) = (-1)^{k-1} \delta_{kn} + \delta_{k0}p_n, \\
  p_k^*(h_n) = h_{n-k},\quad p_k^*(e_n) = (-1)^{k-1}e_{n-k},\quad p_k^*(p_n) = n\delta_{nk} + \delta_{k0}p_n.
\end{gather*}
Then, for example, since $\Delta(e_n)=\sum_{i=0}^n e_i\otimes e_{n-i}$, we have, by Lemma \ref{lem:adjoint-action-on-product},
\[ \ts
  e_n^*(h_kf)=\sum_{i=0}^n e_i^*(h_k)e_{n-i}^*(f) = h_k e_n^*(f) + h_{k-1} e_{n-1}^*(f) \quad \text{for all } f \in \Sy.
\]
Thus $[e_n^*,h_k]=h_{k-1} e_{n-1}^*$.  The other relations are proven similarly.

\subsection{Categorification}

By Corollary~\ref{cor:Hecke-like-strong-dualizing}, the tower $A$ is strong and dualizing.  For $n \in \N$, let $E_n$ (resp.\ $L_n$) be the one-dimensional representation of $A_n$ on which each $T_i$ acts by $-1$ (resp.\ by $q$).  Then $\Delta(E_n) \cong \sum_{i=0}^n E_i \otimes E_{n-i}$ and $\Delta(L_n) \cong \sum_{i=0}^n L_i \otimes L_{n-i}$.  Since $\Hom_{A_n} (L_n, E_n) = 0$ unless $n=0$ or $n=1$ (in which case $E_n$ and $L_n$ are both the trivial module), we have, by ~\eqref{cat-eq:cross},
\[ \ts
  \Res_{L_n} \circ \Ind_{E_k} \cong \nabla \left( \bigoplus_{i=0}^n \Res_{L_i \otimes L_{n-i}} (E_k \otimes -) \right) \cong \left( \Ind_{E_k} \circ \Res_{L_n} \right) \oplus \left( \Ind_{E_{k-1}} \circ \Res_{L_{n-1}} \right),
\]
which is a categorification of the last relation of~\eqref{eq:heisenberg-eh-relations} since, under the isomorphism $\cG(A) \cong \cK(A) \cong \Sy$, the class of the representation $L_n$ corresponds to $h_n$ and the class of $E_k$ corresponds to $e_k$.  By~\eqref{cat-eq:ind}, \eqref{cat-eq:res}, and Lemma~\ref{lem:Hecke-like-d-nonzero}, we have
\begin{align*}
  \Ind_{E_n} \circ \Ind_{E_k} \cong \Ind_{\nabla(E_n \otimes E_k)} \cong \Ind_{\nabla(E_k \otimes E_n)} \cong \Ind_{E_k} \circ \Ind_{E_n}, \quad \text{and} \\
  \Res_{L_n} \circ \Res_{L_k} \cong \Res_{\nabla(L_n \otimes L_k)} \cong \Res_{\nabla(L_k \otimes L_n)} \cong \Res_{L_k} \circ \Res_{L_n},
\end{align*}
which categorifies the first two relations of~\eqref{eq:heisenberg-eh-relations}.

\begin{rem} \label{rem:Hecke-presentations}
  While we have chosen to show how Theorem~\ref{theo:categorification} recovers a categorification of the relations~\eqref{eq:heisenberg-eh-relations}, we could just have easily used it to recover categorifications of the relations~\eqref{eq:heisenberg-other-relations}.  This is an illustration of the fact that Theorem~\ref{theo:categorification} does not rely on a particular presentation of the Heisenberg double $\fh(A)$.
\end{rem}

\begin{rem}
  The special case of Theorem~\ref{theo:Fock-space-properties}\eqref{theo-item:Stone-von-Neumann} for the dual pair $(\Sy,\Sy)$ is known as the \emph{Stone--von Neumann Theorem}.
\end{rem}

\begin{rem}
  Since we have $\cK(A) = \cG(A)$, it follows that ${\mathfrak{h}_\pj} = \fh$ in this case (see Definition~\ref{def:p}).
\end{rem}

%
\section{Hecke algebras at roots of unity} \label{sec:Hecke-unity}
%

We now consider Hecke algebras at a root of unity.  Fix $\ell \in \N_+$ and consider the unital $\C$-algebra $A_n$ with generators and relations as in Section~\ref{subsec:sym-functions}, but with $q$ replaced by a fixed $\ell$th root of unity $\zeta$.  By Corollary~\ref{cor:Hecke-like-strong-dualizing}, $A = \bigoplus_{n \in \N} A_n$ is a strong dualizing tower of algebras.  We refer the reader to \cite[\S3.3]{LLT96} for an overview of some of the facts about Grothendieck groups stated in this section.

Let $\mathcal{J}_\ell \subseteq \Sy$ be the ideal generated by the power sum symmetric functions $p_\ell,p_{2\ell},p_{3\ell},\dotsc$, and let $\mathcal{J}_\ell^\perp$ be its orthogonal complement relative to the standard inner product on $\Sy$ (see Section~\ref{subsec:sym-functions}).  Then there are isomorphisms of Hopf algebras
\[
  \cK(A) \cong \mathcal{J}_\ell^\perp \quad \text{and} \quad \cG(A) \cong \Sy/\mathcal{J}_\ell.
\]
Moreover, under these identifications, the inner product between $\cK(A)$ and $\cG(A)$ is that induced by the standard inner product on $\Sy$.

Recall that a partition $\lambda$ is said to be \emph{$\ell$-regular} if each part appears fewer than $\ell$ times.  The specialization $\bar S_\lambda$ of the Specht module $S_\lambda$, $\lambda \in \partition(n)$, to $q = \zeta$ is, in general, no longer an irreducible $A_n$-module.  However, it was shown in \cite[\S6]{DJ86} that if $\lambda$ is $\ell$-regular, then $\bar S_\lambda$ contains a unique maximal submodule $\rad \bar S_\lambda$.  As $\lambda$ varies over the $\ell$-regular partitions of $n$, $D_\lambda := \bar S_\lambda/ \rad \bar S_\lambda$ varies over a complete set of nonisomorphic irreducible representations of $A_n$.  It follows that a basis of $\cG(A)$ (resp.\ $\cK(A)$) is given by the $[D_\lambda]$ (resp.\ $[P_\lambda]$, where $P_\lambda$ is the projective cover of $D_\lambda$) as $\lambda$ varies over the set of $\ell$-regular partitions.  In theory, one could compute the relations in $\fh(A)$ in these bases by using the results of \cite{LLT96} to express the basis elements in terms of the standard symmetric functions and then use the relations in Section~\ref{subsec:classical-h}.  In this way, one would obtain a presentation of $\fh(A)$.  Of course, in general, this presentation would be far from minimal.

In fact, it turns out that $\fh(A)$ is an integral from of the usual Heisenberg algebra $\fh(\Sy,\Sy)$ (see Section~\ref{subsec:classical-h}).  This can be seen as follows.  Recall that the set of power sum functions $p_\lambda$, $\lambda \in \partition$, is an orthogonal basis of $\Sy_\Q=\Q\otimes_\Z\Sy$.  (Throughout we use a subscript $\Q$ to denote extension of scalars to the rational numbers.)  Therefore $\mathcal{J}_{\ell,\Q}$ has a basis given by the set
\[
  \{p_\lambda \ |\  \ell \text{ divides } \lambda_i \text{ for at least one } i \},
\]
and $\mathcal{J}_{\ell,\Q}^\perp$ has a basis
\[
  \{p_\lambda \ |\ \ell \text{ does not divide } \lambda_i \text{ for any } i  \}.
\]
Similarly, $(\Sy/\mathcal{J}_\ell)_\Q$ has a basis
\[
  \{p_\lambda + \mathcal{J}_\ell \ |\ \ell \text{ does not divide } \lambda_i \text{ for any } i  \}.
\]

\begin{rem}
  We see from the above that $\cG(A_n)$ and $\cK(A_n)$ have bases indexed, on the one hand, by the set of $\ell$-regular partitions of $n$ and, on the other hand, by the set of partitions of $n$ in which no part is divisible by $\ell$.  A correspondence between these two sets of partitions is given by Glaisher's Theorem (see, for example, \cite[p.~538]{Le46}).
\end{rem}

For $m \in \N_+$ such that $\ell$ does not divide $m$, let $q_m = p_m + \mathcal{J}_\ell$.  Then we have algebra isomorphisms
\begin{align*}
  \mathcal{J}_{\ell,\Q}^\perp &\cong \Q[p_m\ |\ m \in \N_+,\ \ell \text{ does not divide } m], \quad \text{and} \\
  (\Sy/\mathcal{J}_\ell)_\Q &\cong \Q[q_m\ |\ m \in \N_+,\ \ell \text{ does not divide } m].
\end{align*}
Thus, $\fh(A)_\Q$ is generated by $\{p_m,q_m \ |\  m \in \N_+,\ \ell \text{ does not divide } m \}$ subject to the relations $[p_m,p_n]=[q_m,q_n]=0$ and $[p_m,q_n]=m\delta_{m,n}1$.  It follows that $\fh(A)_\Q$ is isomorphic as an algebra to the classical Heisenberg algebra $\fh(\Sy,\Sy)_\Q$.  Thus we have the following proposition.

\begin{prop} \label{prop:Heis-unity}
  The Heisenberg double associated to the tower of Hecke algebras at a root of unity is an integral form of the classical Heisenberg algebra:
  \[
    \fh(A)_\Q \cong \fh(\Sy,\Sy)_\Q.
  \]
\end{prop}

By Proposition~\ref{prop:Heis-unity}, as $\ell$ varies over the positive integers, we obtain a family of integral forms of the classical Heisenberg algebra.  It would be interesting to work out minimal presentations of these integral forms over $\Z$.  Furthermore, the Cartan map $\cK(A) \to \cG(A)$ is known to have a nonzero determinant (see \cite[Cor.~1]{BK02}).  Therefore, it induces an isomorphism $\cG_\pj(A)_\Q \cong \cG(A)_\Q$, which implies that ${\mathfrak{h}_\pj}(A)_\Q \cong \fh(A)_\Q$.  It is not known whether ${\mathfrak{h}_\pj}(A) \cong \fh(A)$.

It is known that the category $A\pmd$ yields a categorification of the basic representation of $\widehat{\mathfrak{sl}}_n$ via $i$-induction and $i$-restriction functors (see~\cite[p.~218]{LLT96}).  Theorem~\ref{theo:categorification} provides a categorification of the principle Heisenberg subalgebra of $\widehat{\mathfrak{sl}}_n$.

%
\section{0-Hecke algebras} \label{sec:qsym-nsym}
%

We now specialize the constructions of Sections~\ref{sec:general-construction} and~\ref{sec:dual-pairs-from-towers} to the tower of 0-Hecke algebras of type $A$.  We begin by recalling some basic facts about the rings of quasisymmetric and noncommutative symmetric functions.  We refer the reader to~\cite{LMvW13} for further details.

\subsection{The quasisymmetric functions}

Let $\QS$ be the algebra of \emph{quasisymmetric functions} in the variables $x_1,x_2,\dotsc$ over $\Z$.  Recall that this is the subalgebra of $\Z \llbracket x_1,x_2,\dotsc \rrbracket$ consisting of shift invariant elements.  That is, $f \in \QS$ if and only if, for all $k \in \N_+$, the coefficient in $f$ of the monomial $x_1^{n_1} x_2^{n_2} \dotsm x_k^{n_k}$ is equal to the coefficient of the monomial $x_{i_1}^{n_1} x_{i_2}^{n_2} \dotsm x_{i_k}^{n_k}$ for all strictly increasing sequences of positive integers $i_1 < i_2 < \dotsb < i_k$ and all $n_1,n_2,\dotsc,n_k \in \N$.  The algebra $\QS$ is a graded algebra:
\[ \ts
  \QS = \bigoplus_{n \ge 0} \QS_n,
\]
where $\QS_n$ is the $\Z$-submodule of $\QS$ consisting of homogeneous elements of degree $n$.  We adopt the convention that $\QS_n = 0$ for $n < 0$.

The algebra $\QS$ has a basis consisting of the \emph{monomial quasisymmetric functions} $M_\alpha$, which are indexed by compositions $\alpha=(\alpha_1,\dotsc,\alpha_r) \in \comp$:
\[ \ts
  M_\alpha=\sum_{i_1<\dotsb<i_r}x_{i_1}^{\alpha_1} \dotsm x_{i_r}^{\alpha_r}.
\]
We adopt the convention that $M_\varnothing = 1$.

The algebra $\QS$ has another important basis, the \emph{fundamental quasisymmetric functions} $F_\alpha$, which are defined as follows (see~\cite[\S2]{Ges84}).  For two compositions $\alpha,\beta$, write $\beta \preceq \alpha$ if $\beta$ is a refinement of $\alpha$.  For example, $(1,2,1) \preceq (1,3)$.  Then set
\[ \ts
  F_\alpha=\sum_{\beta \preceq \alpha}M_\beta,\quad \alpha \in \comp.
\]

The algebra $\QS$ is, in fact, a graded connected Hopf algebra.  To describe the coproduct, we introduce a bit of notation relating to compositions.  For two compositions $\alpha=(\alpha_1,\dotsc,\alpha_r)$ and $\beta=(\beta_1,\dotsc,\beta_s)$ let $\alpha\cdot\beta=(\alpha_1,\dotsc,\alpha_r,\beta_1,\dotsc,\beta_s)$ and $\alpha \odot \beta=(\alpha_1,\dotsc,\alpha_r+\beta_1,\dotsc,\beta_s)$.  So, for example, if $\alpha=(1,2,1)$ and $\beta=(3,5)$, then $\alpha\cdot\beta=(1,2,1,3,5)$ and $\alpha\odot\beta=(1,2,4,5)$.  Then the coproduct on $\QS$ is given by either of the two following formulas:
\begin{gather*} \ts
  \Delta(M_\alpha)=\sum_{\alpha=\beta\cdot\gamma}M_\beta\otimes M_\gamma, \\ \ts
  \Delta(F_\alpha)=\sum_{\alpha=\beta\cdot\gamma \text{ or }\alpha=\beta\odot\gamma}F_\beta\otimes F_\gamma.
\end{gather*}

Note that naturally $\Sy \subseteq \QS$.  In particular, the monomial symmetric functions can be handily expressed in terms of the monomial quasisymmetric functions:
\begin{equation} \label{eq:m-M-relation} \ts
  m_\lambda=\sum_{\widetilde{\alpha}=\lambda}M_\alpha, \text{ where $\widetilde{\alpha}$ is the partition obtained by sorting $\alpha$}.
\end{equation}

\subsection{The noncommutative symmetric functions}

Define $\NS$, the algebra of \emph{noncommutative symmetric functions}, to be the free associative algebra (over $\Z$) generated by the alphabet $\bh_1,\bh_2,\dotsc$.  Thus $\NS$ has a basis given by $\bh_\alpha := \bh_{\alpha_1} \dotsm \bh_{\alpha_r}$, $\alpha \in \comp$.  This is a graded algebra:
\[ \ts
  \NS = \bigoplus_{n \ge 0} \NS_n,
\]
where $\NS_n = \Span \{\bh_\alpha \ |\ \alpha \in \comp(n)\}$.  We adopt the convention that $\NS_n = 0$ for $n < 0$.

The \emph{noncommutative ribbon Schur functions} $\br_\alpha$ are defined to be
\[ \ts
  \br_\alpha=\sum_{\alpha \preceq \beta}(-1)^{\ell(\alpha)-\ell(\beta)}\bh_\beta,\quad \alpha \in \comp.
\]
These basis elements multiply nicely:
\[ \ts
  \bh_\alpha\bh_\beta=\bh_{\alpha\cdot\beta} \quad \text{and} \quad \br_\alpha\br_\beta=\br_{\alpha\cdot\beta}+\br_{\alpha\odot\beta}.
\]
In fact, $\NS$ is a graded connected Hopf algebra.  The coproduct is given by the formula
\begin{equation} \label{eq:bh-coproduct} \ts
  \Delta(\bh_n)=\sum_{i=0}^n\bh_i\otimes\bh_{n-i}.
\end{equation}

\subsection{The 0-Hecke algebra and its Grothendieck groups}

Let $\F$ be an arbitrary field and let $A_n$ be the unital $\F$-algebra with generators and relations as in Section~\ref{subsec:sym-functions}, but with $q$ replaced by $0$ (i.e.\ the 0-Hecke algebra).  Consider the tower of algebras $A = \bigoplus_{n \in \N} A_n$.  The irreducible $A_n$-modules are all one-dimensional and are naturally enumerated by the set $\comp(n)$ of compositions of $n$ (see \cite[\S3]{Nor79} and \cite[\S5.2]{KT97}).  Let $L_\alpha$ be the irreducible module corresponding to the composition $\alpha \in \comp(n)$ and let $P_\alpha$ be its projective cover.  We then have (see \cite[Cor.~5.8 and Cor.~5.11]{KT97} -- while the statements there are for the case that $\F = \C$, the proofs remain valid over more general fields)
\begin{gather}
  H^- = \cK(A) \cong \NS,\ [P_\alpha] \mapsto \br_\alpha, \\
  H^+ = \cG(A) \cong \QS,\ [L_\alpha] \mapsto F_\alpha. \label{eq:groth-QS}
\end{gather}
We also have
\[
  \cG_\pj(A) \cong \Sy,
\]
and the Cartan map $\cK(A) \to \cG_\pj(A)$ of Definition~\ref{def:G-proj} corresponds to the projection of Hopf algebras
\begin{equation} \label{eq:chi-def}
  \chi \colon \NS \twoheadrightarrow \Sy,\quad \bh_\alpha \mapsto h_{\widetilde{\alpha}}.
\end{equation}
Alternatively, it is given by $\chi(\br_\alpha)=r_\alpha$, where $r_\alpha$ is the usual ribbon Schur function.
This is a reformulation of \cite[Prop.~5.9]{KT97}.

The bilinear form~\eqref{eq:tower-pairing} becomes the well-known perfect Hopf pairing of the Hopf algebras $\QS$ and $\NS$ given as follows:
\begin{gather*}
  \left<\cdot,\cdot\right> \colon \NS\times \QS \to \Z, \\
  \left<\bh_\alpha,M_\beta \right>=\delta_{\alpha\beta}=\left<\br_\alpha,F_\beta \right>,\quad \alpha,\beta \in \comp.
\end{gather*}
In this way, $(\NS,\QS)$ is a dual pair of Hopf algebras.

\subsection{The quasi-Heisenberg algebra} \label{subsec:quasi-Heisenberg}

We now apply the construction of Section~\ref{subsec:h-definition} to the dual pair $(\QS,\NS)$.

\begin{defin}[(Projective) quasi-Heisenberg algebra]
  We call $\fq := \fh(\QS,\NS)$ the \emph{quasi-Heisenberg algebra}.  We define the \emph{projective quasi-Heisenberg algebra} $\fq_\pj$ to be the subalgebra of $\fq$ generated by $\NS$ and $\Sy \subseteq \QS$ (see Definition~\ref{def:p}).
\end{defin}

\begin{lem} \label{LemmaRelations}
  In $\fq$ we have, for all $\alpha=(\alpha_1,\dotsc,\alpha_r) \in \comp$, $n \in \N_+$,
  \[
    \left[\pR{\bh}_n^*,M_\alpha\right] = M_{(\alpha_1,\dotsc,\alpha_{r-1})}{\pR{\bh}_{n-\alpha_r}^*},
  \]
  with the understanding that $\pR{\bh}_k^*=0$ for $k<0$.
\end{lem}

\begin{proof}
  By Lemma~\ref{lem:adjoint-action-on-product} and~\eqref{eq:bh-coproduct}, we have
  \[ \ts
    \pR{\bh}_n^*(M_\alpha G) = \sum_{i=0}^n\pR{\bh}_i^*(M_\alpha)\pR{\bh}_{n-i}^*(G).
  \]
  So, if $n\geq \alpha_r$, we have $\pR{\bh}_n^*(M_\alpha G) = M_{\alpha} \pR{\bh}_n^*(G) + M_{(\alpha_1,\dotsc,\alpha_{r-1})} \pR{\bh}_{n-\alpha_r}^*(G)$.  The result follows.
\end{proof}

\begin{cor} \label{lem:h-m-comm-rel}
  For $n \in \N$ and $\lambda \in \partition$, we have
  \begin{equation} \ts
    \label{eq:comm-rel-Rh-m} [\pR{\bh}_n^*, m_\lambda] = \sum_{j=1}^n m_{\lambda - j} \pR{\bh}_{n-j}^*,
  \end{equation}
  where $\lambda - j$ is equal to the partition obtained from removing a part $j$ from $\lambda$ if $\lambda$ has such a part and $m_{\lambda-j}$ is defined to be zero otherwise.  In particular, for $n,k \in \N$, we have
  \[ \ts
    [\Rh_n^*, p_k] = \Rh_{n-k}^*,\quad [\Rh_n^*, e_k] = e_{k-1} \Rh_{n-1}^*,\quad [\Rh_n^*, h_k] = \sum_{j=1}^n h_{k-j} \Rh_{n-j}^*.
  \]
\end{cor}

\begin{proof}
  Equation~\eqref{eq:comm-rel-Rh-m} follows from~\eqref{eq:m-M-relation} and Lemma~\ref{LemmaRelations}.  The remainder of the relations then follow by expressing $p_k$, $e_k$ and $h_k$ in terms of the monomial symmetric functions $m_\lambda$.
\end{proof}

\begin{cor} \label{cor:q-Mgenerators}
  The quasi-Heisenberg algebra $\fq$ is generated by the set
  \[
    \{M_\alpha,\pR{\bh}_n^*\ |\ \alpha \in \comp,\ n \in \N_+\}.
  \]
  The $M_\alpha$ multiply as in $\QS$ (for a precise description of this product, see \cite[\S3.3.1]{LMvW13}) and
  \begin{gather*} \ts
    \left[\pR{\bh}_n^*,M_\alpha\right] = M_{(\alpha_1,\dotsc,\alpha_{r-1})}\pR{\bh}_{n-\alpha_r}^*,\quad \alpha = (\alpha_1,\dotsc,\alpha_r) \in \comp,\ n \in \N_+.
  \end{gather*}
\end{cor}

\begin{rem} \label{rem:q-Fgenerators}
  We also note that the algebra $\fq$ has generators
  \[
    \{F_\alpha,\pR{\bh}_n^*\ |\ \alpha \in \comp,\ n \in \N_+\}.
  \]
  From a representation theoretic point of view, these are more natural since the $F_\alpha$ correspond to simple $A_n$-modules (see~\eqref{eq:groth-QS}).  The $F_\alpha$ then multiply as in $\QS$ (for a precise description of this product, see \cite[\S3.3.1]{LMvW13}) and
  \begin{equation} \label{eq:0-Hecke-simple-relation} \ts
    \left[ \pR{\bh}_n^*,F_\alpha \right] = \sum_{i=1}^{\alpha_r} F_{(\alpha_1,\dotsc,\alpha_{r}-i)}{\pR{\bh}_{n-i}^*},\quad \alpha = (\alpha_1,\dotsc,\alpha_r) \in \comp,\ n \in \N_+.
  \end{equation}
\end{rem}

\begin{rem}
  Note that the above presentations are far from minimal.  There are polynomial generators of $\QS$, enumerated by elementary Lyndon words (see \cite[Th.~6.7.5]{HGK10}), which one could use instead of the $M_\alpha$ in the above presentation.  This would result in a minimal presentation of $\fq$.
\end{rem}

The following result gives a presentation of the projective quasi-Heisenberg algebra in terms of generators and relations.

\begin{prop} \label{prop:p-presentation}
  The algebra $\fq_\pj$ is generated by the set
  \[
    \{e_n, \pR{\bh}_n^*\ |\ n \in \N\}.
  \]
  The relations are
  \[
    [e_n, e_k] = 0,\quad [\Rh_n^*, e_k] = e_{k-1} \Rh_{n-1}^*,\quad n,k \in \N.
  \]
\end{prop}

\begin{proof}
  This follows immediately from the definition of $\fq_\pj$ and Corollary~\ref{lem:h-m-comm-rel}.
\end{proof}

\begin{rem} \label{rem:fp-to-fh-map}
  Note the similarity of the presentation of Proposition~\ref{prop:p-presentation} to the presentation of the usual Heisenberg algebra $\fh(\Sy,\Sy)$ given in~\eqref{eq:heisenberg-eh-relations}.  The only difference is that the $h_n^*$ commute, whereas the $\pR{\bh}_n^*$ do not.  There is a natural surjective map of algebras $\fq_\pj \to \fh(\Sy,\Sy)$ given by $e_n \mapsto e_n$, $\pR{\bh}_n^* \mapsto h_n^*$, $n \in \N_+$.
\end{rem}

\subsection{Fock spaces and categorification} \label{subsec:quasi-Fock-spaces}

As described in Section~\ref{subsec:general-Fock-space}, the quasi-Heisenberg algebra $\fq$ acts naturally on $\QS$ and we call this the \emph{lowest weight Fock space representation} of $\fq$.  By Theorem~\ref{theo:Fock-space-properties}\eqref{theo-item:Stone-von-Neumann}, any representation of $\fq$ generated by a lowest weight vacuum vector is isomorphic to $\QS$.

Similarly, as in Definition~\ref{def:p-Fock-space}, the projective quasi-Heisenberg algebra $\fq_\pj$ acts naturally on $\Sy$ and we call this the \emph{lowest weight Fock space representation} of $\fq_\pj$.  As a $\fq_\pj$-module, $\Sy$ is generated by the lowest weight vacuum vector $1 \in \Sy$.  By Proposition~\ref{prop:p-Fock-space-properties}\eqref{prop-item:p-Stone-von-Neumann}, any representation of $\fq_\pj$ generated by a lowest weight vacuum vector is isomorphic to $\Sy$.  However, this representation is not faithful since it factors through the projection from $\fq_\pj$ to the usual Heisenberg algebra (see Remark~\ref{rem:fp-to-fh-map}).  On the other hand, the highest weight Fock space representation of $\fq_\pj$ is faithful (see Proposition~\ref{prop:hw-p-FS-faithful}).

By Remark~\ref{rem:Hecke-like-general-field}, $A$ is a strong dualizing tower of algebras.  Therefore, Theorem~\ref{theo:categorification} yields a categorification of the Fock space representations of $\fq$ and $\fq_\pj$.  For instance, it is straightforward to verify that
\begin{gather*} \ts
  \Delta(L_\alpha) \cong \bigoplus_{\alpha=\beta\cdot\gamma \text{ or }\alpha=\beta\odot\gamma} L_\beta \otimes L_\gamma \quad \text{for all } \alpha \in \comp, \\ \ts
  \Psi^{\otimes 2} \Delta \Psi^{-1}(P_{(n)}) \cong \Delta(P_{(n)}) \cong \bigoplus_{i=0}^n P_{(i)} \otimes P_{(n-i)} \quad \text{for all } n \in \N.
\end{gather*}
For $\alpha = (\alpha_1,\dotsc,\alpha_r) \in \comp$ and $i \in \{0,\dotsc,\alpha_r\}$, it follows that $\Res_{P_{(i)}} L_\alpha = L_{(\alpha_1,\dotsc,\alpha_r-i)}$.  Thus we have
\[ \ts
  \Res_{P_{(n)}} \circ \Ind_{L_\alpha} \cong \nabla \left( \bigoplus_{i=0}^n \Res_{P_{(i)} \otimes P_{(n-i)}} (L_\alpha \otimes -) \right) \cong \bigoplus_{i=0}^{\alpha_r} \Ind_{L_{(\alpha_1,\dotsc,\alpha_r-i})} \Res_{P_{(n-i)}},
\]
which is a categorification of the relation~\eqref{eq:0-Hecke-simple-relation}.  The categorification of the multiplication of the elements $F_\alpha$, $\alpha \in \comp$, follows from the computation of the induction in $A\md$ (see, for example, the proof of \cite[Prop.~4.15]{DKKT97}).

%
\section{\texorpdfstring{Application: $\QS$ is free over $\Sy$}{Application: QS is free over Sym}}
%

As a final application of the methods of the current paper, we use the generalized Stone--von Neumann Theorem for $\fq_\pj$ (Proposition~\ref{prop:p-Fock-space-properties}) to prove that $\QS$ is free over $\Sy$.  This gives a proof that is quite different from the one previously appearing in the literature.  The previous proof proceeds by constructing a free commutative polynomial basis of $\QS$ enumerated by elementary Lyndon words (see~\cite[Thm~6.7.5]{HGK10}).  Then the freeness of $\QS$ over $\Sy$ follows from the fact that the polynomial basis contains the elementary symmetric functions.  The construction of this polynomial basis turned out to be rather difficult, and many false proofs appeared in the literature.  We refer the reader to \cite{Haz01a,Haz01b} for more on the history of this result.

\begin{lem} \label{lem:fp-comp-reducible}
  Suppose $V$ is a $\fq_\pj$-module which is generated (as a $\fq_\pj$-module) by a finite set of lowest weight vacuum vectors.  Then $V$ is a direct sum of copies of lowest weight Fock space.
\end{lem}

\begin{proof}
  Let $\{v_i\}_{i \in I}$ denote a set of lowest weight vacuum vectors that generates $V$ and such that $I$ has minimal cardinality.  We claim that
  \begin{equation} \label{eq:vi-vj-overlap}
    \Z v_i \cap \Z v_j = \{0\} \quad \text{for all } i \ne j.
  \end{equation}
  Suppose, on the contrary, that $\Z v_i \cap \Z v_j \ne \{0\}$ for some $i \ne j$.  Then $n_i v_i = n_j v_j$ for some $n_i,n_j \in \Z$.  Let $m = \gcd(n_i,n_j)$ and choose $a_i,a_j \in \Z$ such that $m = a_i n_i + a_j n_j$.  Set $w = a_jv_i + a_iv_j$.  Then $w$ is clearly a lowest weight vacuum vector, and we have
  \[
    \frac{n_i}{m} w = \frac{1}{m}(a_j n_iv_i + a_i n_iv_j)=  \frac{1}{m}(a_j n_jv_j + a_i n_iv_j) = v_j.
  \]
  Similarly, $\frac{n_j}{m}w = v_i$.  Thus $\{v_k\}_{k \in I \setminus \{i,j\}} \cup \{w\}$ is a set of lowest weight vacuum vectors that generates $V$, contradicting the minimality of the cardinality of $I$.

  By Proposition~\ref{prop:p-Fock-space-properties}\eqref{prop-item:p-Stone-von-Neumann}, $\fq_\pj \cdot v_i \cong \Sy$ as $\fq_\pj$-modules.  It then follows from Proposition~\ref{prop:p-Fock-space-properties}\eqref{prop-item:p-Fock-space-subreps} and~\eqref{eq:vi-vj-overlap} that $\fq_\pj \cdot v_i \cap \fq_\pj \cdot v_j = \{0\}$ for $i \ne j$.  The lemma follows.
\end{proof}

Define an increasing filtration of $\fq_\pj$-submodules of $\QS$ as follows.  For $n \in \N$, let
\[ \ts
  \QS^{(n)} := \sum_{\ell(\alpha) \leq n}\fq_\pj\cdot M_\alpha.
\]
In particular, note that $\QS^{(0)}=\Sy$.  We adopt the convention that $\QS^{(-1)} = \{0\}$.

\begin{prop} \label{prop:qsym-free-over-sym}
  The space $\QS$ of quasisymmetric functions is free as a $\Sy$-module.
\end{prop}

\begin{proof}
  Note that, for $\alpha \in \comp$ such that $\ell(\alpha)=n$, we have $\pR{\bh}_m^*(M_\alpha) \in \QS^{(n-1)}$ for any $m>0$.  Therefore, in the quotient $V_n=\QS^{(n)}/\QS^{(n-1)}$, such $M_\alpha$ are lowest weight vacuum vectors.  It is clear that these vectors generate $V_n$, and therefore, by Lemma~\ref{lem:fp-comp-reducible},
  \[ \ts
    V_n = \bigoplus_{v \in \mathcal{S}_n}\Sy\cdot v,
  \]
  where $\mathcal{S}_n$ is some collection of vacuum vectors in $V_n$.

  Consider the short exact sequence
  \[
    0 \to \QS^{(n-1)} \to \QS^{(n)} \to V_n \to 0.
  \]
  Since $V_n$ is a free (hence projective) $\Sy$-module, the above sequence splits.   Therefore, if $\QS^{(n-1)}$ is free over $\Sy$, then so is $\QS^{(n)}$.

  By the argument in the previous paragraph we can choose nested sets of vectors in $\QS$
  \[
    \widetilde{\mathcal{S}}_0 \subseteq \widetilde{\mathcal{S}}_1 \subseteq \widetilde{\mathcal{S}}_2 \subseteq \dotsb
  \]
  such that, for every $n \in \N$, we have $\QS^{(n)} = \bigoplus_{\tilde{v} \in \widetilde{\mathcal{S}}_n}\Sy\cdot \tilde v$.  Let $\widetilde{\mathcal{S}} = \bigcup_{n \in \N} \widetilde{\mathcal{S}}_n$.  Then
  \[ \ts
    \QS = \bigoplus_{v \in \widetilde{\mathcal{S}}}\Sy\cdot v.  \qedhere
  \]
\end{proof}


\bibliographystyle{alpha}
\bibliography{Heisenberg-double-biblist}

\end{document}